%% file: an-homotopical-description-of-small-presheaves.tex
\documentclass[a4paper]{amsart}

\usepackage{cmbright}
\usepackage{mathrsfs}
\usepackage[utf8]{inputenc}
\usepackage[T1]{fontenc}
\usepackage[english]{babel}
\usepackage{amsfonts}
\usepackage{amsthm}
\usepackage{amssymb}
\usepackage{mathcomp}
\usepackage{mathrsfs}
\usepackage[bb=boondox]{mathalfa}
\usepackage{graphicx}
\usepackage{epstopdf}
\usepackage[centertags]{amsmath}
\usepackage[pdftex]{hyperref}
\usepackage{vmargin}
\usepackage{tikz}
\usepackage{tikz-cd}
\usepackage{dsfont}
\usepackage{cases}
\usepackage{mathtools}
\usepackage{shuffle}
\usepackage{adjustbox}

\include{macros}
\include{tree}

\title{An homotopical description of small presheaves}
\author{Brice Le Grignou}
\email{bricelegrignou@gmail.com}

\date{\today}



\begin{document}

\maketitle

\begin{abstract}
    This article describes the cocompletion of a category $C$ with finite limits as the homotopy category of some equivalence 2-groupoids in coproducts of elements of $C$. This yields a simple link between several definitions of an infinitary pretopos.
\end{abstract}

\setcounter{tocdepth}{1}
\tableofcontents

\section*{Introduction}

Properties of categories may sometimes be encoded as structures on them. An important example is cocompleteness. Indeed, being cocomplete is a property of a category that may also be encoded as the structure of an algebra over the 2-monad of small presheaves $\Psh$ that acts on the 2-category $\mathfrak{Cat}$ of categories and that sends a category $\categ C$ the category of small presheaves on $\categ C$ that is functors
$$
F \in \Fun{\categ C^\op}{\Set}
$$
that are left Kan extensions of functors from a small category. Another related example is the property to have small coproduct; it is also encoded by the 2-monad $\mathcal S$ that transform a category $\categ C$ into the category $\mathcal S (\categ C)$ made up of sets $X$ equipped with a functor $X \to \categ C$. These two 2-monads as well as many monads that encode categories that have some property are lax idempotent monads (\cite{KellyLack}) also called KZ doctrines. Roughly, a KZ-doctrine $M$ is a 2-monad whose whole structure (functoriality, monad product) is deduced from the unit maps
$$
X \to M(X)
$$
by some left Kan extension.

Both 2-monad $\Psh$ and $\mathcal{S}$, that act on the 2-category $\mathfrak{Cat}$ of categories, restricts to the sub-2-category $\mathfrak{Lex}$ of categories with finite limits and left exact functors. In that context, the $\mathcal S$-algebras become categories with finite limits and small coproducts that satisfy some coherence conditions, namely
\begin{enumerate}
    \item coproducts are disjoint;
    \item coproducts are universal.
\end{enumerate}
Moreover, $\Psh$-algebras become $\mathcal S$-algebras 
that have coequalisers of equivalence groupoids (and then are cocomplete) that satisfy some coherence conditions, namely
\begin{enumerate}
    \item equivalence groupoids are effective;
    \item effective epimorphisms are stable through pullbacks.
\end{enumerate}
Such categories are called pretopos.

In this paper, we fill the gap between $\mathcal{S}$ and $\Psh$ in the context of categories with finite limits. Precisely, we describe in concrete terms a KZ-doctrine $\mathcal E_h^{(2)}$ (whose existence is shown for instance in \cite{GarnerLack}) on $\mathfrak{Lex}$ whose algebras are categories with finite limits and coequalisers of equivalence groupoids and so that equivalence groupoids are effective and effective epimorphisms are stable through pullbacks. It is built as follows: from a category $\categ C$, on can build an homotopy category $\mathcal{E}_h(\categ C)$ of equivalence groupoids in $\categ C$ whose
\begin{enumerate}
    \item objects $X$ are equivalence groupoids, that is pairs of objects $(X_0,X_1)$ of $\categ C$ equipped with a map $X_1 \to X_0 \times X_0$ that makes $X_1$ an equivalence relation on $X_0$;
    \item morphisms $f: X \to Y$ are equivalence classes of maps $f_0 : X_0 \to Y_0$ so that the composition
    $$
    X_1 \to X_0 \times X_0 \xrightarrow{f_0 \times f_0} Y_0 \times Y_0
    $$
    factorises through $Y_1$; two such maps $f,g: X \to Y$ being equivalent if
    $$
    (f_0, g_0) : X_0 \to Y_0 \times Y_0
    $$
    factorises through $Y_1$.
\end{enumerate}
The construction $\categ C \mapsto \mathcal{E}_h(\categ C)$ is not a KZ doctrine; this is actually what we call a KZ-machine. This means that $\mathcal{E}_h$ has most attributes of a KZ-doctrine, which allow to talk about $\mathcal{E}_h$-algebra; however, $\mathcal{E}_h(\categ C)$ is not in general a $\mathcal{E}_h$-algebra. To build a KZ doctrine, one needs to go one step further. Indeed, $\mathcal{E}_h^{(2)}(\categ C)$ is the full subcategory of $\mathcal{E}_h\mathcal{E}_h(\categ C)$ spanned by equivalence groupoids $(X_0, X_1)$ so that $X_0$ belong to the subcategory $\categ C \subset \mathcal{E}_h(\categ C)$. 

To conclude, adding colimits to a category $\categ C$ may be thinked of as adding diagrams targeting $\categ C$. This perspective is for instance concretised
by the category of categories fibered in sets over $\categ C$, that is equivalent to presheaves through the Grothendieck construction. In this paper, we concretise this perspective in another way that has an homotopical flavour and in the context of categories with finite limits.

\subsection*{Layout}

In the first section, we introduce the notion of a KZ-machine and recall that of a KZ-doctrine. In the second section, we recall the KZ-doctrine of sums and that of small presheaves on the 2-category of categories and on the 2-category of categories with finite limits. In the third section, we define the KZ-machine of equivalence groupoids and the KZ-doctrine of equivalence 2-groupoids. Finally, the fourth section rebuild the KZ-doctrine of small presheaves by stacking that of sums and that of equivalence 2-groupoids.

\subsection*{Notations and convention}

\begin{enumerate}
    \item Let $\mathcal{U} < \mathcal{U}$ be two universes. An element of $\mathcal{U}$ will be called a small set or just a set. An element of $\mathcal{V}$ is called large and a subset of $\mathcal{V}$ is called very large.
    \item Without further precision, a category will mean a locally small category whose set of object is a subset of $\mathcal{U}$.
    \item We denote $\Delta$ the category of standard simplicies, that is the full subcategory of the category of categories spanned by the posets
    $$
    [n] = 0 < 1 < \cdots < n, \quad n \in \mathbb N.
    $$
    We denote $\Delta_{\leq 1}$ the full subcategory of $\Delta$ spanned by $[0]$ and $[1]$.
    We also use standard notations for simplicial sets (see for instance \cite{GoerssJardine})
    \item Let us consider morphisms $f, g: X \to Y$ and morphisms $f',g' : X' \to Y$ in a category and let us suppose that their exists an implicit order on the pair $\{f,g\}$ so that $f < g$ (this happens for instance if $f$ is denoted $h_0$ and $g$ is denoted $h_1$) and an implicit order on the pair $\{f',g'\}$ so that $f' < g'$. Then, the formula $X \times_Y X'$ will describe the pullback of the cospan diagram
    $$
    X \xrightarrow{g} Y \xleftarrow{f'} X' .
    $$
    Moreover, we get a map $(f,g'): X \times_Y X' \to Y \times Y$.
    
\end{enumerate}

\section{Kock-Zoberlein machines and doctrines}

The goal of this section is to introduce the notion of a KZ-machine and to recall (and redefine from KZ-machines) the notion of a KZ-doctrine on a 2-category.

\subsection{Two-categories}

\begin{definition}
For us, a 2-category will mean a bicategory. A strict 2-category will mean a bicategory whose associativity and unitality are strict; that is a category enriched in categories.
\end{definition}

\begin{definition}
Let $\mathfrak{Cat}$ be the very large (locally large) strict 2-category of categories and let $\mathfrak{Lex}$ be the very large (locally large) strict 2-category whose objects are lex categories (that is categories with finite limits), morphisms are left exact (lex for short) functors and 2-morphisms are natural transformations.
\end{definition}

\begin{definition}\label{defKan}
Let $\mathfrak{C}$ be a 2-category. and let us consider a span diagram
$$
Y \xleftarrow{p} X \xrightarrow{f} Z .
$$
The left Kan extension $p_! (f)$ of $f$ along $p$ is, if it exists, the initial object of the category
$$
\mathfrak{C}(Y, Z) \times_{\mathfrak{C}(X, Z)} \mathfrak{C}(X, Z)_{f/} .
$$
More concretely, this is the data of a morphism $p_! (f) : Y \to Z$ and a 2-morphism $f \to p_! (f) \circ p$ so that, for any morphism $g: Y \to Z$ and any 2-morphism $f \to g \circ p$, there exists a unique 2-morphism $p_!(f) \to g$ so that the previous 2-morphism factors as:
$$
f \to p_! (f) \circ p \to g \circ p .
$$
\end{definition}

\begin{definition}
A left Kan extension $p_!(f)$ is called absolute if (using the notations of the previous Definition \ref{defKan}), for any morphism $g: Z \to Z'$, then $g \circ p_!(f)$ is the left Kan extension of $g \circ f$ along $p$:
$$
p_{!}(g \circ f) = g \circ p_!(f) .
$$
\end{definition}

\begin{definition}
An adjunction in a 2-category $\mathfrak{C}$ is the data of two objects $X,Y$ a morphism $L: X \to Y$ called the left adjoint, a morphism $R : Y \to X$ called the right adjoint, and two 2-morphisms $\eta : \id_{X} \to RL$ and $\epsilon : LR \to \id_Y$ called respectively the unit and the counit and so that the two following maps are identites of morphisms
\begin{align}
    &L \to LRL \to L;
    \\
    &R \to RLR \to R.
\end{align}
\end{definition}

\begin{example}
Let us consider an adjunction $L \dashv R$ in a 2-category $\mathfrak C$. Then $R$ is the left Kan extension of $\id$ along $L$ and it is an absolute Kan extension.
\end{example}

\begin{lemma}\label{lemmaadjKan}
Let us consider a morphism $L: X \to Y$ that has a right adjoint $R$ in such a way that the counit $LR \to \id_Y$ is invertible. Let us also consider a cospan diagram
$$
V \xleftarrow{p} U \xrightarrow{f} Y
$$
so that $Rf$ has a left Kan extension $p_!(Rf)$. Then, the morphism $L \circ p_!(Rf)$ is the left Kan extension of $LR f \simeq f$ along $p$.
\end{lemma}

\begin{proof}
Let us consider a morphism $g : V \to Y$ and a 2-morphism $f \to gp$. We have a sequence of canonical natural isomorphisms
\begin{align*}
    \hom_{\mathfrak C(V,Y)}(L p_!(Rf), g)
    & = \hom_{\mathfrak C(V,X)}(p_!(Rf), R g) 
    \\
    & = \hom_{\mathfrak C(U,X)}(Rf, Rgp)
    \\
    & = \hom_{\mathfrak C(U,Y)}(LRf, gp)
    \\
    & = \hom_{\mathfrak C(U,Y)}(f, gp).
\end{align*}
This proves the result.
\end{proof}

\begin{lemma}\label{lemmaadjfromkan}
Let us consider a morphism $R : X \to Y$ and a left Kan extension $L$ of $\id_X$ along $R$ so that the 2-morphism
$$
\id_X \to LR
$$
is invertible. Let us also suppose that $\id_Y$ is the left Kan extension of $R$ along $R$. Then $L$ is left adjoint to $R$ in such a way that the counit is the previous isomorphism.
\end{lemma}

\begin{proof}
The canonical isomorphism $R \simeq RLR$ induces a map by left Kan extension
$$
\id_{Y} = R_!(R) \to RL,
$$
which is the unit of the adjunction.
It is clear that the composite map
$$
R \simeq RLR \to R
$$
is the identity of $R$.

Besides, let us consider the composite map $L \to  LRL \to L$. Precomposing by $R$ yields a composite map $LR \to  LRLR \to LR$ which is actually equal to the composition
$$
LR \xrightarrow{\epsilon^{-1}} \id_X \xrightarrow{\epsilon} LR
$$
and is thus the identity of $LR$. Since $L$ is a left Kan extension of $LR \simeq \id_X$ along $R$, the previous map $L \to  LRL \to L$ is also the identity.
\end{proof}

\subsection{Kock-Zoberlein machines}

\begin{definition}
A Kock-Zolerbein (KZ) machine $P$ on a 2-category $\mathfrak C$ is the data of, for any object $X$, an object $P(X)$ and a map $i_P(X): X \to P(X)$ so that for any morphism $f:X \to Y$, the composition $X \to Y \to P(Y)$ has a left Kan extension $P(f)$ along $i_P(X)$ and so that the map $i_P(Y) \circ f \to P(f) \circ i_P(X)$ is an isomorphism.
$$
\begin{tikzcd}
    P(X)
    \ar[r, "P(f)"]
    & P(Y)
    \\
    X
    \ar[u, "{i_P(X)}"] \ar[r,"f"']
    & Y .
    \ar[u, "{i_P(Y)}"'] \ar[lu, Rightarrow, "\sim"']
\end{tikzcd}
$$
Moreover, we require that for any two morphisms $f: X \to Y$ and $g : Y \to Z$, the canonical map
$$
P(g \circ f) \to P(g) \circ P(f) 
$$
is an isomorphism and that for any object $X$, the map
$$
P(\id_X) \to \id_{P(X)}
$$
is an isomorphism
\end{definition}

\begin{remark}
A KZ machine $P$ on $\mathfrak{C}$ is in particular a pseudo-functor from $\mathfrak{C}$ to $\mathfrak{C}$ and the map $i_P(X) : X \to P(X)$ is natural with respect to $X$.
\end{remark}

\begin{definition}
Given two KZ-machines $P,Q$, we say that $P$ is simpler than $Q$ if for any morphism $f:X \to Y$, the morphism $i_Q(Y) \circ f: X \to Q(Y)$ has a left Kan extension $i_{P,Q}(f)$ along $i_P(X)$ so that the map
$$
i_{P,Q}(f) \circ i_{P}(X) \to i_{Q}(X) \circ f
$$
is an isomorphism. Moreover, for any two morphisms $f: X \to Y$ and $g: Y \to Z$, the natural map
\begin{align*}
i_{P,Q}(g \circ f)  & \to i_{P,Q}(g)  \circ P(f)
\\
i_{P,Q}(g \circ f)  &\to Q(g)  \circ i_{P,Q}(f) 
\end{align*}
are isomorphisms. In particular for any morphism $f$, we have canonical isomorphisms
$$
i_{P,Q}(Y) \circ P(f) \simeq i_{P,Q}(f) \simeq Q(f) \circ i_{P,Q}(X).
$$
\end{definition}

\begin{definition}
Two KZ-machines $P,Q$ are equivalent
if for any $X$, there exists an equivalence $i_{P,Q}(X) : P(X) \to Q(X)$ and a 2-isomorphism
$$
i_{P,Q}(X) \circ i_P(X) \simeq i_Q(X) .
$$
\end{definition}

\begin{proposition}
If $P,Q$ are equivalent, then $P$ is simpler than $Q$, $Q$ is simpler than $P$. Moreover, for any composable morphisms $f :X \to Y$ and $g : Y \to Z$, the canonical 2-morphisms 
\begin{align*}
    &Q(g \circ f) \to i_{P,Q}(g) \circ i_{Q,P}(f)
    \\
    &P(g \circ f) \to i_{Q,P}(g) \circ i_{P,Q}(f)
\end{align*}
are invertible.
\end{proposition}

\begin{proof}
Straightforward.
\end{proof}

\subsection{Algebras over a KZ-machine}

\begin{definition}
Let $P$ be a KZ-machine. An object $Y$ is called $P$-cocomplete (or a $P$-algebra) if any morphism $f: X \to Y$ has a left Kan extension $\overline{f}$ along $i_M(X)$ so that the 2-morphism
$$
f \to \overline{f} \circ i_M(X)
$$
is invertible. Moreover, we require that for any map $g: X' \to X$, the map
$$
\overline{f \circ g} \to \overline{f} \circ P(g)
$$
is an isomorphism.
$$
\begin{tikzcd}
    X
    \ar[r, "f"]\ar[d, "i_P(X)"']
    \ar[rd, Rightarrow, shorten >=18pt, "\sim"]
    & Y
    \\
    P(X)
    \ar[ru, bend right, "\overline{f}"']
    & {}
\end{tikzcd}
$$
\end{definition}

\begin{definition}
For a $P$-cocomplete object $Y$, we denote
$$
m_P(Y) = \overline{\id_Y}.
$$
Then, for any morphism $f : X \to Y$, the map $\overline{f} \to m_P(X) \circ  P(f)$ is an isomorphism.
\end{definition}

\begin{proposition}\label{proppcomp}
Let $X$ be an object of a 2-category $\mathfrak C$ and let $P$ be a KZ machine.
The two assertions are equivalent:
\begin{enumerate}
    \item $X$ is $P$-cocomplete;
    \item the map $i_p(X): X \to P(X)$ has a left adjoint $m_P(X) : P(X) \to X$ so that the counit is invertible.
\end{enumerate}
\end{proposition}

\begin{proof}
Let us assume that $X$ is $P$-cocomplete. The fact that $m_P(X)$ is left adjoint to $i_P(X)$ with invertible counit is a direct consequence of Lemma \ref{lemmaadjfromkan}.

Conversely, let us assume that $R = i_P(X)$ has a left adjoint $L$ and that the counit $LR \to \id$ is invertible. Let us consider a map $f: U \to X$. Lemma \ref{lemmaadjKan} tells us that $f$ has a left Kan extension along $i_P(U)$, which is $L \circ P(f)$. It is then straightforward to show from this fact that $X$ is $P$-cocomplete.
\end{proof}

\begin{definition}
A morphism $f: X \to Y$ between two $P$-cocomplete objects is $P$-linear if 
the 2-morphism
$$
 \overline f \to f \circ m_P(X)
$$
is an isomorphism.
\end{definition}

\begin{lemma}\label{lemmadefplinear}
A morphism $f: X \to Y$ between two $P$-cocomplete objects is $P$-linear if and only if for any morphism $g :U \to X$, the map
$$
\overline{f \circ g} \to f \circ \underline g
$$
is an isomorphism.
\end{lemma}

\begin{proof}
It is clear that a morphism $f$ that satisfies such a property is $P$-linear. Conversely, let us suppose that $f$ is $P$-linear and let $g: U \to X$ be a morphism. The map $\overline{f \circ g} \to f \circ \underline g$ decomposes as
$$
\overline{f \circ g} \simeq \overline f \circ P(g)  \to f \circ m_P(X) \circ P(g) \simeq f \circ \overline{g}.
$$
The middle map is invertible since $f$ is $P$-linear. So the map $\overline{f \circ g} \to f \circ \underline g$ is also invertible.
\end{proof}

\begin{definition}
We define the 2-category $\mathfrak{Alg}_{\mathfrak{C}, P}$ whose objects are $P$-cocomplete objects in $\mathfrak{C}$, morphisms are $P$-linear morphisms and 2-morphisms are those of $\mathfrak{C}$ between $P$-linear morphisms.
\end{definition}

\begin{definition}
Let $P,Q$ be two KZ-machines. One says that $P$ is Morita simpler than $Q$ if $\mathfrak{Alg}_{\mathfrak{C}, P}$ contains $\mathfrak{Alg}_{\mathfrak{C}, Q}$. In other words
\begin{enumerate}
    \item if an object $X$ is $Q$-cocomplete, then it is $P$-cocomplete;
    \item if a morphism $f: X \to Y$ is $Q$-linear, then it is $P$-linear.
\end{enumerate}
\end{definition}

\begin{proposition}\label{propsimpler}
Let $P,Q$ be two KZ-machines. If $P$ is simpler than $Q$, then $P$ is Morita simpler than $Q$.
\end{proposition}

\begin{proof}
Let $X$ be a $Q$-cocomplete object. Then, let us denote $f$ the composite morphism 
$$
P(X) \xrightarrow{i_{P,Q}(X)} Q(X) \xrightarrow{m_Q(X)} X .
$$
Since $i_{P,Q}(X)$ is the left Kan extension of $i_Q(X)$ along $i_P(X)$ and since $m_Q(X)$ is left adjoint to $i_Q(X)$ in such a way that the counit is invertible, Lemma \ref{lemmaadjKan} tells us that $f$ is the left Kan extension of $\id_X$ along $i_P(X)$. Then, Lemma \ref{lemmaadjfromkan} tells us that $f$ is left adjoint to $i_P(X)$ so that the counit is invertible. Finally, the fact that $X$ is $P$-cocomplete follows from Proposition \ref{proppcomp}.

Besides, proving that a $Q$-linear morphism $f :X \to Y$ is $P$-linear is straightforward using the definition.
\end{proof}

\begin{definition}
Two KZ machines $P,Q$ are called Morita equivalent if they share the same algebras and the same linear morphisms, that is if an object $X$ is $P$-cocomplete if and only if it is $Q$-cocomplete and if a morphism $f: X \to Y$ between $P$-cocomplete objects is $P$-linear if and only if it is $Q$-linear.
\end{definition}

\subsection{Composition of KZ-machines}

Let $P,Q$ be two KZ-machines on a 2-category $\mathfrak{C}$. One can notice that the following square diagram is commutative up to a canonical 2-isomorphism
$$
\begin{tikzcd}
    P(X)
    \ar[r,"i_Q(P(X))"]
    & QP(X)
    \\
    X
    \ar[r,"i_Q(X)"'] \ar[u,"i_P(X)"]
    & Q(X)
    \ar[u, "Q(i_P(X))"']
\end{tikzcd}
$$
for any object $X$ in $\mathfrak{C}$.

\begin{definition}
Let us denote
$$
i_{QP}(X) = i_Q(P(X)) \circ i_P(X)    
$$
for any object $X$.
\end{definition}

\begin{definition}
One says that $Q$ extends $P$ if for any morphism $f:X \to Y$, the morphism $i_Q(P(Y))$ preserves the left Kan extension $P(f)$ in the sense that the canonical isomorphism
 $$
 i_{QP}(Y) \circ f = i_Q(P(Y)) \circ i_P(Y) \circ f \xrightarrow{\simeq} i_Q(P(Y)) \circ P(f) \circ i_P(X) 
 $$
 makes $i_Q(P(Y)) \circ P(f)$ the left Kan extension of  $i_{QP}(Y) \circ f$ along $i_P(X)$.
\end{definition}

\begin{definition}
One says that the pair $(Q,P)$ is composable if the construction
$$
X \in \mathfrak C \mapsto Q(P(X))
$$
equipped with the morphisms $i_{QP}(X) : X \to Q(P(X))$
is a KZ-machine denoted $QP$ so that for any morphism $f: X \to Y$, we have
$$
QP(f) = Q(P(f)).
$$
In other words, $Q(P(f))$ is the left Kan extension of $i_{QP}(Y) \circ f$ along $i_{QP}(X)$.
\end{definition}

\begin{proposition}\label{propcomposekzpreserve}
 Let us suppose that the pair $(Q,P)$ is composable. Then $Q$ is simpler than $QP$ through the maps
 $$
 i_{Q, QP}(f) : Q(X) \xrightarrow{Q(f)} Q(Y) \xrightarrow{Q(i_P(X))} QP(X) ,
 $$
 for any morphism $f: X \to Y$.
\end{proposition}

\begin{proof}
It suffices to notice that the map
 $$
 Q(i_P(Y)) \circ Q(f) \to Q(P(f)) \circ Q(i_P(X))
 $$
 is the left Kan extension of
 $$
 i_{QP}(Y) \circ f = i_{Q}(P(Y)) \circ i_P(Y) \circ f
 $$
 along $i_Q(X)$.
\end{proof}

\begin{proposition}\label{propcomposekzextends}
 Let us suppose that $Q$ extends $P$. Then
 the pair $(Q,P)$ is composable. Moreover, $P$ is simpler than $QP$ through the maps
 $$
 i_{P,QP}(f) : P(X) \xrightarrow{P(f)} P(Y) \xrightarrow{i_Q(P(Y))} QP(Y) ,
 $$
 for any morphism $f: X\to Y$.
\end{proposition}

\begin{proof}
To prove that the pair $(Q,P)$ is composable,
it suffices to prove that $Q(P(f))$ is the left Kan extension of the composite map
$$
X \to Y \to QP(Y)
$$
along $i_{QP}(X)$. For that purpose, let us contemplate the following diagram
$$
\begin{tikzcd}
    & QP(X)
    \ar[rd, "Q(P(f))"]
    \\
    P(X)
    \ar[r,"P(f)"] \ar[ru, "{i_Q(P(X))}"]
    & P(Y)
    \ar[r, "{i_Q(P(Y))}"]
    &QP(Y)
    \\
    X
    \ar[rr, "f"'] \ar[u]
    &&Y
    \ar[u, "i_{QP}(Y)"']
\end{tikzcd}
$$
The map $P(X) \to QP(Y)$ is the left Kan extension of the map $i_{QP}(Y)\circ f : X \to QP(Y)$ along $i_P(X)$ and $Q(P(f))$ is the left Kan extension of this left Kan extension along $i_Q(P(X))$. Hence $Q(P(f))$ is the left Kan extension of $i_{QP}(Y)\circ f$ along $i_{QP}(Y) = i_P(P(X)) \circ i_P(X)$.

Finally, the fact that $P$ is simpler than $QP$ follows from the definition of the fact that $Q$ extends $P$.
\end{proof}

\begin{proposition}
Let us assume that the pair $(Q,P)$ is composable. Let us consider an object $X$ and a morphism $f: X \to Y$. Then
\begin{enumerate}
    \item if $X$ is $QP$-cocomplete if it is $P$-cocomplete and $Q$-cocomplete; 
    \item if $f$ is $QP$-linear if it is $P$-linear and $Q$-linear.
\end{enumerate}
One can replace the implications ("if") by equivalences ("if and only if") in the case where $Q$ and extends $P$.
\end{proposition}

\begin{proof}
First, let us suppose that $X$ is $P$-cocomplete and that it is $Q$-cocomplete. Since $Q$ induces a pseudo-functor, then $Q(m_P(X))$ is left adjoint to $Q(i_P(X))$ and the related counit is invertible. Subsequently, $m_Q(X) \circ Q(m_P(X))$ is left adjoint to $Q(i_P(X)) \circ i_Q(X)$ and the related counit is also invertible. So $X$ is $QP$-cocomplete.

Then, let us suppose that $f$ is $P$-linear and $Q$-linear. The composite map
$$
\begin{tikzcd}
    m_{QP}(Y) \circ QP(f)
    \ar[d, equal]
    \\
    m_{Q}(Y) \circ Q(m_P(Y)) \circ QP(f)
    \ar[d, "\simeq"]
    \\
    m_{Q}(Y) \circ Q(f) \circ Q(m_P(X))
    \ar[d, "\simeq"]
    \\
    f \circ m_{Q}(X) \circ Q(m_P(X))
    \ar[d, equal]
    \\
    f \circ m_{QP}(X)
\end{tikzcd}
$$
is an isomorphism as the composition of isomorphisms. So $f$ is $QP$-linear.

The converse assertions in the case where $Q$ extends $P$ follows from the fact that $Q$ and $P$ are simpler than $QP$ (Proposition
\ref{propcomposekzpreserve} and Proposition \ref{propcomposekzextends}).
\end{proof}

\subsection{Kock-Zoberlein doctrines}

\begin{definition}\cite{KellyLack}
A KZ doctrine (also called a lax idempotent monad) $M$ on a 2-category is a KZ machine so that
\begin{enumerate}
    \item for any object $X$, $M(X)$ is $M$-cocomplete;
    \item for any morphism $f: X \to Y$, the morphism $M(f) : M(X) \to M(Y)$ is $M$-linear.
\end{enumerate}
\end{definition}

\begin{remark}
The definition given above is not the usual definition of a KZ-doctrine (see for instance \cite{Walker}). The following proposition \ref{propositionkzdoctrine} shows that the two definitions (the usual one and this one) are equivalent.
\end{remark}

\begin{remark}
The name KZ doctrine comes from the fact that it is in particular a doctrine (that is a 2-monad). For instance, the canonical map
$$
MM(X) \to M(X) ,
$$
is just the left Kan extension of $\id_{M(X)}$ along $i_M(M(X)): M(X) \to MM(X)$.
Moreover, the $M$-cocomplete objects are exactly the $M$-algebras (in the sense of alagebras over a 2-monad). See \cite{MarmolejoWood} for more details.
\end{remark}

Let $M$ be a KZ-doctrine on a 2-category $\mathfrak C$.

\begin{lemma}
For any $M$-cocomplete object $X$, the map $m_P(X) : M(X) \to X$ is $M$-linear.
\end{lemma}

\begin{proof}
Let us denote
\begin{align*}
    L & = m_M(X);
    \\
    R &= i_M(X);
    \\
    L_M &= m_M(M(X));
    \\
    R_M &= i_M(M(X)).
\end{align*}
One can notice that $L$ and $L_M$ are left adjoint to respectively $R$ and $R_M$. Since $M$ is a pseudofunctor, $M(L)$ is left adjoint to $M(R)$. Moreover, composing adjunctions, we obtain that $LL_M$ is left adjoint to $R_MR$ and $LM(L)$ is left adjoint to $M(R)R$. The counits of all these adjunctions are invertible.

Our goal is to show that the canonical 2-morphism $\alpha : LM(L) \to LL_M$ induced by the fact that $LM(L)$ is the left Kan extension of $L$ along $R_M$ and by the canonical isomorphism
$$
LM(L)R_M \simeq LRL \simeq L \simeq LL_M R_M
$$
is an isomorphism. One can check that $\alpha$ decomposes as
$$
LM(L) \to LM(L)R_ML_M \simeq LRLL_M \simeq LL_M .
$$
But the related 2-morphism from $R_MR$ to $M(R)R$ is the canonical isomorphism
$$
R_MR \simeq M(R)R.
$$
So $\alpha$ is invertible.
\end{proof}

\begin{proposition}\label{propositionkzdoctrine}
For any $M$-cocomplete object $Y$ and any morphism $f: X \to Y$, the map $\overline f: M(X) \to Y$ is $M$-linear.
\end{proposition}

\begin{proof}
It suffices to notice that $\overline{f} = m_M(X) \circ M(f)$ and that $m_M(X)$ and $M(f)$ are $M$-linear maps.
\end{proof}

\begin{proposition}\label{propfreealg}
Let $Y$ be a $M$-cocomplete object and let $X$ be an object. Then, the adjunction
$$
\begin{tikzcd}
\mathfrak{C}(X, Y)
\ar[rr, shift left, "f \mapsto \overline{f}"]
&&  \mathfrak{Alg}_{\mathfrak C, M}(M(X), Y)
\ar[ll, shift left, "- \circ i_M(X)"]
\end{tikzcd}
$$
is an equivalence adjunction, that is the counit and the unit are invertible.
\end{proposition}

\begin{proof}
By definition of the construction $f \mapsto \overline{f}$, the counit is invertible. The invertibility of the unit follows from the characterisation of an $M$-linear morphism given in Lemma \ref{lemmadefplinear}.
\end{proof}

\begin{definition}
Given a KZ machine $P$, a rectification of $P$ is a KZ doctrine $M$ that is Morita equivalent to $P$.
\end{definition}

\begin{proposition}\label{proprecsimp}
Let $M$ be a KZ-doctrine and let $P$ be a KZ-machine. If $P$ is Morita simpler than $M$, then it is simpler than $M$.
\end{proposition}

\begin{proof}
For any morphism $f: X \to Y$, since $M(Y)$ is $P$-cocomplete, one gets a map $i_{P,M}(f): P(X) \to M(Y)$ as the left Kan extension
$$
i_{P,M}(f) = \overline{i_M(Y) \circ f} = \overline{i_M(Y)} \circ P(f)
$$
of $i_M(Y) \circ f$ along $i_P(X)$. For any composable morphisms $f: X \to Y$ and $g: Y \to Z$, then the canonical maps
\begin{align*}
i_{P,M}(g \circ f) &\to i_{P,M}(g) \circ P(f)
\\
i_{P,M}(g \circ f) &\to M(g) \circ  i_{P,M}(f)
\end{align*}
are invertible by $P$-cocompleteness of $M(Z)$ for the first one and by $P$-linearity of $M(g)$ for the second one.
\end{proof}

\subsection{Morita equivalence of KZ-doctrines}

\begin{proposition}\label{propsimpsimpimplieseq}
Let $M,N$ be two KZ-doctrines on a 2-category $\mathfrak{C}$.
The three following assertions are equivalent:
\begin{enumerate}
    \item $M$ and $N$ are Morita equivalent;
    \item $M$ is simpler than $N$ and $N$ is simpler than $M$;
    \item $M$ and $N$ are equivalent.
\end{enumerate}
\end{proposition}

\begin{proof}
First it is clear that (3) implies (1) (and also (2)). We know from Proposition \ref{proprecsimp} and from Proposition \ref{propsimpler} that (1) and (2) are equivalent. Finally, let us assume (2) and (1) and let us prove (3). Let $X$ be an object. By Proposition \ref{propfreealg}, the maps $i_{M,N}(X): M(X) \to N(X)$ and $i_{N,M}(X): N(X) \to M(X)$ are $M$-linear and $N$-linear. Since
\begin{align*}
    i_{N,M}(X) &\circ i_{M,N}(X) \circ i_M(X) \simeq i_M(X);
    \\
    i_{M,N}(X) &\circ i_{N,M}(X) \circ i_N(X) \simeq i_N(X) ; 
\end{align*}
then Proposition \ref{propfreealg} provides us with 2-isomorphisms
\begin{align*}
    i_{N,M}(X) &\circ i_{M,N}(X) \simeq \id_{M(X)};
    \\
    i_{M,N}(X) &\circ i_{N,M}(X) \simeq \id_{N(X)}.
\end{align*}
So $M$ and $N$ are equivalent.
\end{proof}


\section{KZ doctrines and KZ machines on the 2-categories of categories and lex categories}

In this section, we recall the KZ-doctrine $\mathcal{S}$ of sums and that of small presheaves $\Psh$ on the 2-category $\mathfrak{Cat}$ of categories and on the 2-category $\mathfrak{Lex}$ of lex categories.

See for instance \cite{GarnerLack} for more details.

\subsection{Lex KZ-machines and pointwise KZ-machines}

\begin{definition}
A KZ-machine $P$ on the 2-category $\mathfrak{Cat}$ or on the 2-category $\mathfrak{Lex}$ is called pointwise if for any morphism $f: X \to Y$ the left Kan extension $P(f)$ is pointwise and, in the case where $Y$ is $P$-cocomplete, the left Kan extension $\overline f$ is pointwise.
\end{definition}

Left Kan extensions in $\mathfrak{Lex}$ are often already left Kan extensions in $\mathfrak{Cat}$.

\begin{lemma}\label{lemmafromcattolex}
Let $p: \categ C \to \categ C'$ and $f : \categ C \to \categ D$ be two lex functors. Let us suppose that $f$ has a left Kan extension along $p$ in the 2-category $\mathfrak{Cat}$. If the functor $p_!(f)$ is lex, then it defines a left Kan extension of $f$ along $p$ in the 2-category $\mathfrak{Lex}$
\end{lemma}

\begin{proof}
Straightforward.
\end{proof}

\begin{definition}
Let us consider a KZ machine $P$ on $\mathfrak{Cat}$. Then $P$ is called a lex machine in for any lex category $\categ C$, $P(\categ C)$ is lex. Moreover, $i_P(\categ C)$ is lex and for any lex functor $f : C \to D$, $P(f)$ is lex.
\end{definition}

\begin{definition}
Let $M$ be a KZ doctrine on $\mathfrak{Cat}$ that is also a lex machine. They, $M$ is called a lex doctrine if for any lex category $\categ C$, the canonical left Kan extension
$$
m_M(M(\categ C)) : MM(\categ C) \to M(\categ C)
$$
is lex.
\end{definition}

\begin{proposition}\label{proplexmachine}
A lex KZ-machine $P$ on $\mathfrak{Cat}$ induces canonically a KZ-machine on $\mathfrak{Lex}$ that we also denote $P$. Then, $P$-cocomplete objects in $\mathfrak{Lex}$ are in particular lex categories that are $P$-cocomplete in $\mathfrak{Cat}$. Moreover, a lex functor $f: \categ C \to \categ D$ between $P$-cocomplete objects in $\mathfrak{Lex}$ is $P$-linear in $\mathfrak{Lex}$ if and only if it is $P$-linear in $\mathfrak{Cat}$.
\end{proposition}

\begin{proof}
This is a direct consequence of Lemma \ref{lemmafromcattolex}.
\end{proof}

\begin{proposition}
Let $M$ be a KZ doctrine on $\mathfrak{Cat}$ that is also a lex machine. If $M$ is a lex doctrine, then the induced KZ-machine $M$ on $\mathfrak{Lex}$ is a KZ-doctrine.
\end{proposition}

\begin{proof}
For every lex category $\categ C$, the functor
$$
m_M(M(\categ C)) : MM(\categ C) \to M(\categ C)
$$
is both lex and left adjoint to $i_M(M(\categ C))$ in $\mathfrak{Cat}$ so that the counit is invertible. So it left adjoint to $i_M(M(\categ C))$ in $\mathfrak{Lex}$ so that the counit is invertible, and so $M(\categ C)$ is $M$-cocomplete in $\mathfrak{Lex}$.

For every lex functor $f : \categ C \to \categ D$, $M(f)$ is lex and it is $M$-linear in $\mathfrak{Cat}$. So by Proposition \ref{proplexmachine}, it is $M$-linear in $\mathfrak{Lex}$.
\end{proof}

\subsection{The KZ-doctrine of coproducts}

\subsubsection{In categories}

\begin{definition}\label{definensemble}
Let $\categ C$ be a category.
The category $\mathcal{S}(\categ C)$ of $\categ C$-sums is
the category obtained from $\categ C$ by freely adding small coproducts. More concretely,
\begin{itemize}
    \itemt a $\categ C$-sum $X$ is the data of a set $\mathcal{I}_X$ (called the set of indices of $X$) and of a functor
    $$
    i \in \mathcal{I}_X \mapsto X(i) \in \categ C ;
    $$
    \itemt a morphism $f$ of a $\categ C$-sums from $X$ to $Y$ is the data of a function $f_{\mathcal{I}} : \mathcal{I}_X \to \mathcal{I}_Y$ and a natural transformation $X \to Y\circ f_{\mathcal{I}}$ given by maps
    $$
    f(i) : X(i) \to Y(f_{\mathcal{I}}(i)),\  i \in \mathcal{I}_X.
    $$
\end{itemize}
\end{definition}

The construction $\categ C \in \mathfrak{Cat} \mapsto \mathcal{S}(\categ C)$ together with the canonical fully faithful embedding $\categ C \to \mathcal S (\categ C)$ makes $\mathcal{S}$ a pointwise KZ-doctrine.

Let $\categ C$ be a category. The three following assertions are equivalent
\begin{enumerate}
    \item the functor $\categ C \to \mathcal{S}(\categ C)$ has a left adjoint;
    \item $\categ C$ is $\mathcal{S}$-cocomplete in $\mathfrak{Cat}$;
    \item $\categ C$ admits small coproducts.
\end{enumerate}

Moreover, a functor $f: \categ C \to \categ D$ between categories with small coproducts is $\mathcal{S}$-linear if and only if it preserves these coproducts.

\subsubsection{In lex categories}

\begin{lemma}
The KZ-doctrine $\mathcal{S}$ on $\mathfrak{Cat}$ is a lex doctrine.
\end{lemma}

\begin{proof}
Let $\categ C$ be a lex category. Let us show that $\mathcal{S}(\categ C)$ is lex.
The final $\categ C$-sum is given by the final set and the final element of $\categ C$. Moreover, for any cospan diagram of $\categ C$-sums $X \to Z \leftarrow Y$, its pullback is the $\categ C$-sum $X \times_Z Y$ whose set of indices is $\mathcal{I}_X\times_{\mathcal{I}_Z} \mathcal{I}_Y$ and so that for any element $(i,j)$ in this set
$$
(X \times_Z Y)(i,j) = X(i) \times_{Z(k)} Y(j)
$$
where $k$ denotes the common image of $i$ and $j$ in $\mathcal{I}_Z$.

It is then straightforward to check that the functors $\categ C \to \mathcal{S}(\categ C)$ and $\mathcal{S}\mathcal{S}(\categ C) \to \mathcal{S}(\categ C)$ are lex and that $\mathcal{S}(f)$ is lex for any lex functor $f$.
\end{proof}

\begin{lemma}
Let $\categ C$ be a lex category that has small coproducts. Then $\categ C$ is $\mathcal S$-cocomplete in $\mathfrak{Lex}$ if and only if the two following conditions are satisfied:
\begin{enumerate}
    \item coproducts are universal in the sense that the map
    $$
    \coprod_i (X_i \times_Z \to Y) \to (\coprod_i X_i) \times_Z Y
    $$
    is an isomorphism for any cospan diagram $\coprod_i X_i \to Z \leftarrow Y$;
    \item coproducts are disjoint in the sense that the map
    $$
    \emptyset \to X \times_{X \sqcup Y} Y
    $$
    is an isomorphism for any objects $X,Y$.
\end{enumerate}
\end{lemma}

\begin{proof}
Straightforward.
\end{proof}

\subsection{The KZ-doctrine of small presheaves}

\subsubsection{Small presheaves of categories}

Let $\categ C$ be a category and let $F$ be a presheaf on $\categ C$, that is a functor from $\categ C^\op$ to $\Set$. The following assertions are equivalent:
\begin{enumerate}
    \item $F$ is the left Kan extension of a functor $g : \categ C'^\op \to \Set$ from a small category $\categ C'^\op$ along a functor $p : \categ C'^\op \to \categ C^\op$;
    \item $F$ is the colimit in the large category $\Fun{\categ C}{\Set}$ with small colimits of a small diagram with values in $\categ C$ (seen as a full subcategory of $\Fun{\categ C}{\Set}$ through the Yoneda embedding);
    \item there exists a small category $\categ D$ and a cofinal functor
    $$
    \categ D \to \categ C/F .
    $$
\end{enumerate}

\begin{definition}
A presheaf $F$ on a category $\categ C$ that satisfies the above assertions is called a small presheaf.
Let $\Psh(\categ C)$ be the (locally small) category of small presheaves.
\end{definition}

\begin{remark}
If $\categ C$ is not small but just locally small, then the category of all presheaves is large. However, the full subcategory of small presheaves is locally small.
\end{remark}

\begin{proposition}\cite[Theorem 5.35]{kelly}
The construction $\categ C \mapsto \Psh(\categ C)$ together with the Yoneda fully faithful embedding
$$
i_{\Psh}(\categ C) :\categ C \mapsto \Psh(\categ C)
$$
form a pointwise KZ-doctrine on $\mathfrak{Cat}$ denoted $\Psh$. Moreover the $\Psh$-cocomplete categories are the cocomplete categories and the $\Psh$-linear morphisms are the functors that preserve colimits.
\end{proposition}

\begin{corollary}
Let $X$ be an object of $\categ C$. Then the functor
$$
F \in \Psh(\categ C) \mapsto F(X) = \hom_{\Psh(\categ C)}(X, F)
$$
preserves colimits.
\end{corollary}

\begin{proof}
It suffices to notice that it is the pointwise left Kan extension of the Yoneda functor
$$
\hom_{\categ C}(X, -): \categ C \to \Set.
$$
\end{proof}

\subsubsection{Small presheaves of lex categories}

\begin{proposition}\cite[Proposition 4.3 and Remark 6.6]{DayLack}
Let $\categ C$ be a category and let $f$ be a functor.
\begin{enumerate}
    \item If $\categ C$ is lex, then so is $\Psh(C)$. Moreover, the functors
    \begin{align*}
    i_{\Psh}(\categ C) &: \categ C \to \Psh(\categ C);
    \\
    m_{\Psh}(\Psh(\categ C)) &: \Psh(\Psh(\categ C)) \to \Psh(\categ C)
    \end{align*}
    are lex.
    \item If $f$ is lex, then so is $\Psh(f)$.
\end{enumerate}
In other words, the KZ-doctrine $\Psh$ on $\mathfrak{Cat}$ is a lex doctrine.
\end{proposition}

One can notice that for any lex category $\categ C$ and any object $X \in \categ C$, the functor
$$
\hom_{\Psh(\categ C)}(X,-) : \Psh(\categ C) \to \Set
$$
is lex.

We know from \cite[Proposition 2.5]{GarnerLack} and \cite[Corollary 3.3]{Kock}, that a $\Psh$-cocomplete objects are exactly infinitary pretopos (see Definiton \ref{definitionpretopos} below). The proof uses mainly the concept of a postulated colimit. We will actually prove again this result.


\section{The doctrine of equivalence 2-groupoids}

In this section, we introduce a KZ-machine $\mathcal{E}_h$ and a KZ-doctrine $\mathcal{E}_h^{(2)}$ on the 2-category $\mathfrak{Lex}$ of lex categories (that is categories with finite limits), that both encode lex categories so that
\begin{enumerate}
    \item equivalence groupoids are effective;
    \item effective epimorphisms are stable through pullback.
\end{enumerate}
Let $\categ C$ be a lex category (that is a category that admits finite limits).

\subsection{Simplicial objects}

\begin{definition}
Let $\sSet_f$ be the category of finite simplicial sets, that is simplicial sets that are obtained as finite colimits of standard simplicies $\Delta[n]$. The functor
$$
\categ{Fun}_{\mathrm{lex}}\left( \sSet_f^\op, \categ C \right) \to \Fun{\Delta^\op}{\categ{C}} = \categ C^{\Delta^\op}
$$
where "lex" designates left exact functors,
is an equivalence of categories.
\end{definition}

One can notice that the full subcategory $\sSet_f \subset \sSet$ is stable through finite products since this products commutes with colimits in both variables and since finite products of standard simplicies are finite simplicial sets.

\begin{definition}
Let $A \in \categ  C^{\Delta^\op}$ be a simplicial object in $\categ C$ and let $X \in \sSet_f$ be a finite simplicial set. Then, we define $A^{X}$ as the simplicial object in $\categ C$ whose induced left exact functor from $\sSet_f$ to $\categ C$ is
$$
(A^{X})(Y)  = A(X \times Y)
$$
using the fact that $X \times Y$ is still a finite simplicial set. This construction induces a functor
$$
\categ C^{\Delta^\op} \times \sSet_f^\op \to \categ C^{\Delta^\op}
$$
that makes $\categ C^{\Delta^\op}$ cotensored over $\sSet_f$.
\end{definition}

For any finite diagram of finite simplicial sets $F : I \to \sSet$ and any simplicial object $A \in \categ C^{\Delta^\op}$, one has
$$
A^{(\varinjlim_I F)} = \varprojlim_{i \in I^\op} A^{F(i)}.
$$

\begin{definition}
Let $A,B \in \categ C^{\Delta^\op}$ be two simplicial objects. Then, we define $\mathrm{Map}_{\Delta}(A,B)$ as the simplicial set so that
$$
\mathrm{Map}_{\Delta}(A,B)(X) = \hom_{\categ C}(A, B^X)
$$
for any finite simplicial set $X$. This defines an enrichement of $\categ C^{\Delta^\op}$ by $\sSet$.
\end{definition}

\begin{definition}
Let $A$ be an simplicial object. Then we denote $A_{\leq 1}$ the image of $A$ through the forgetful functor 
$$
\sSet = \categ C^{\Delta^\op} \to \categ C^{(\Delta_{\leq 1}^\op)}.
$$
\end{definition}

\begin{definition}
Let $\mathrm{cosk}_1$ be the functor from $\categ C^{\Delta_{\leq 1}^\op}$ to $\categ C^{\Delta^\op}$ that sends an element $A$ to the simplicial object
$$
\mathrm{cosk}_1 (A)(X) = A(X_{\leq 1})
= A\left(\varinjlim_{(\Delta[m], \phi) \in (\Delta_{\leq 1})_{/X}} \Delta[m]\right)
= \varprojlim_{(\Delta[m], \phi) \in ((\Delta_{\leq 1})_{/X})^\op} A_m.
$$
This functor is right adjoint to the forgetful functor
$$
A \mapsto A_{\leq 1} ,
$$
and since the counit of the adjunction is an isomorphism, this functor $\mathrm{cosk}_1$ is fully faithful.
\end{definition}

\begin{definition}
We say that a simplicial object $A$ in $\categ C$ is 1-coskeletal if it belongs to the essential image of $\categ C^{\Delta_{\leq 1}^{\op}}$. The functor $\mathrm{cosk}_1$ is an equivalence from $C^{\Delta_{\leq 1}^\op}$ to 1-cokeletal simplicial objects.
\end{definition}

One can notice that a simplicial object $A$ is 1-coskeletal if and only if
the maps
$$
A^{\Delta[n]} \to A^{\partial \Delta[n]}
$$
is an isomorphism for any $m \geq 2$ and if and only if for any $B$, the simplicial set 
$$
\mathrm{Map}_\Delta(A,B)
$$ is 1-coskeletal.
Moreover, for any 1-coskeletal simplicial object $A = \mathrm{cosk}_1(A')$ and any finite simplicial set, then $A^X$ is 1-coskeletal as
$$
A^X(Y) = A'(X_{\leq 1} \times Y_{\leq 1}) .
$$

\begin{definition}
A simplicial object $B$ in $\categ C$ is Kan if for any $A$, the simplicial set $\mathrm{Map}_{\Delta}(A,B)$ is a Kan complex.
\end{definition}

\begin{lemma}
A 1-coskeletal simplicial object $A$ is Kan if and only if the maps
$$
A(\Delta[2]) = A(\partial\Delta[2]) \to A(\Lambda^i[2]), \quad i=0,1,2;
$$
have sections.
\end{lemma}

\begin{proof}
It is clear that $A$ is Kan if and only if the maps
$$
A^{\Delta[2]} = A^{\partial\Delta[2]} \to A^{\Lambda^i[2]}, \quad i=0,1,2;
$$
have sections. In particular, if $A$ is Kan, then the maps
$$
A(\Delta[2]) = A(\partial\Delta[2]) \to A(\Lambda^i[2]), \quad i=0,1,2;
$$
have sections. Conversely, if these maps have section, we get from these sections maps
$$
A(\Lambda^i[2] \times \Delta[1]) = A(\Lambda^i[2]) \times_{A_0^3} A_1^3 \times_{A_0^3} A(\Lambda^i[2]) \to A(\partial\Delta[2]) \times_{A_0^3} A_1^3 \times_{A_0^3} A(\partial\Delta[2])
= A(\partial\Delta[2] \times \Delta[1]) ,
$$
which are sections of the maps
$$
A(\partial\Delta[2] \times \Delta[1]) \to A(\Lambda^i[2] \times \Delta[1]) .
$$
Combined with the sections of the maps 
$$
A(\partial\Delta[2]) \to A(\Lambda^i[2]),
$$
this gives us sections of the maps of simplicial objects $A^{\partial\Delta[2]} \to A^{\Lambda^i[2]}$.
\end{proof}

\begin{definition}
Given a cospan of Kan simplicial objects $\categ C$
$$
A \xrightarrow{f} B \xleftarrow{g} A',
$$
let $A \times^\Delta_B A'$ to be the following Kan simplicial object
$$
A \times^\Delta_B A' = A \times_B B^{\Delta[1]} \times_B A',
$$
where the limit in the formula is the strict limit in the category  $\categ C^{\Delta^\op}$.
\end{definition}

\begin{lemma}\label{lemmahomotopycolimits}
The $\infty$-category described by the simplicial category of Kan simplicial objects in $\categ C$ has finite limits. Moreover, the full $\infty$-subcategory spanned by $1$-coskeletal objects is stable through these finite limits.
\end{lemma}

\begin{proof}
The final object is just the functor 
$$
\Delta[n] \in \Delta^\op \mapsto \ast_{\categ C}.
$$
Moreover, given a cospan diagram $A \to B \leftarrow A'$, its homotopy pullback is $A \times^\Delta_B A'$.
Finally, $1$-coskeletal objects are stable through these finite limits since they contain the final object and since they are stable through strict limits and cotensorisation with finite simplicial sets.
\end{proof}

\subsection{Equivalence groupoids}

\begin{definition}
A weak equivalence groupoid is a 1-coskeletal Kan simplicial object. We denote $\mathcal{E}^w(\categ C)$ the category of weak equivalence groupoids, that is the full subcategory of $\categ C^{\Delta_{\leq 1}^\op}$ spanned by Kan objects. Moreover, we denote $\mathcal{E}^w_h(\categ C)$ the homotopy category of weak equivalence groupoids, that is the category whose objects are those of 
$\mathcal{E}^w(\categ C)$ and whose morphisms are homotopy classes of morphisms of $\mathcal{E}^w(\categ C)$. Remember that two morphisms $f,g \in \hom_{\mathcal{E}^w(\categ C)}(X,Y)$ are homotopic $f \sim g$ if the map
$$
(f,g) : X \to Y \times Y
$$
factorises through $Y^{\Delta[1]}$.
\end{definition}

\begin{definition}
An equivalence groupoid in $\categ C$ is a weak equivalence groupoid $A$ so that the map $A_1 \to A_0 \times A_0$ is a monomorphism. We denote respectively $\mathcal{E}(\categ C)$ and $\mathcal{E}_h(\categ C)$ the full subcategories of $\mathcal{E}^w(\categ C)$ and $\mathcal{E}^w_h(\categ C)$ spanned by equivalence groupoids.
\end{definition}

Equivalently, an equivalence groupoid $A$ in $\categ C$ is the data of
\begin{itemize}
    \itemt an object $A_0$;
    \itemt a subobject $A_1 \to A_0 \times A_0$ so that for any $X \in \categ C$, the inclusion map
    $$
    \hom_{\categ C}(X, A_1) \to \hom_{\categ C}(X, A_0) \times \hom_{\categ C}(X, A_0) 
    $$
    is an equivalence relation.
\end{itemize}
Equivalently, this is a groupoid object 
$$
\begin{tikzcd}
    A_1 \ar[r, bend left, "s"] \ar[r, bend right, swap, "t"]
    & A_0 \ar[l]
\end{tikzcd}
$$
so that the map $A_1 \to A_0 \times A_0$ is monic.

\begin{proposition}
The category $\mathcal{E}^w_h(\categ C)$ has finite limits. Moreover, the full subcategories
$$
\categ C \hookrightarrow \mathcal{E}_h(\categ C) 
\hookrightarrow \mathcal{E}^w_h(\categ C) 
$$
are stable through finite limits.
\end{proposition}

\begin{proof}
The canonical functor from the simplicial category of weak equivalence groupoids to  $\mathcal{E}^w_h(\categ C)$ is an equivalence of simplicial categories since all the mapping spaces are 1-coskeletal. Hence, by Lemma \ref{lemmahomotopycolimits}, $\mathcal{E}^w_h(\categ C)$ has finite limits. The rest of the proposition follows from a straightforward checking.
\end{proof}

\begin{proposition}\label{lemmahomotopy}
Two morphisms of equivalence groupoids $f,g : A \to B$ are homotopic (and hence represent the same morphism in $\mathcal{E}_h(\categ C)$) if and only if the map
$$
(f_0,g_0) : A_0 \to B_0 \times B_0
$$
factorises through $B_1$.
\end{proposition}

\begin{proof}
On the one hand, let us assume that $f \sim g$. Then $(f,g)$ factorises through $B^{\Delta[1]}$. In particular, $(f_0,g_0)$ factorises through $B^{\Delta[1]}_0 = B_1$.

On the other hand, let us assume that $(f_0,g_0)$ factorises through $B_1$. This gives us a factorisation of $(f_1,g_1)$
through $B(\partial \square)$
where $\partial \square$ is the 1-skeletal simplicial set represented by the following graph
$$
\begin{tikzcd}
    10
    \ar[r]
    & 11
    \\
    00
    \ar[r] \ar[u]
    & 01 .
    \ar[u]
\end{tikzcd}
$$
Since $B$ is an equivalence groupoid, the map $(B^{\Delta[1]})_1 = B(\Delta[1] \times \Delta[1]) \to B(\partial \square)$ is an isomorphism. Thus we get a map $A_1 \to (B^{\Delta[1]})_1$. Combined with the map $A_0 \to B_1 = (B^{\Delta[1]})_0$, it forms a morphism
$
A \to B^{\Delta[1]}
$
that is an homotopy linking $f$ and $g$.
\end{proof}

\begin{proposition}\label{cormonicmap}
Let $f : A \to B$ be a morphism of equivalence groupoids. The homotopy class $[f] \in \mathcal{E}_h(\categ C)(A,B)$ is monic if and only if the map
$$
A_1 \to A_0  \times_{B_0} B_1 \times_{B_0} A_0
$$
is an isomorphism.
\end{proposition}

\begin{proof}
Let us denote $X = A_0  \times_{B_0} B_1 \times_{B_0} A_0$.

On the one hand, let us suppose that $[f]$ is monic.
The two projections from $X$ onto $A_0$ gives two morphisms $(g,g')$ towards $A$ so that $fg \sim fg'$. Since $[f]$ is monic, $g \sim g'$ and we get a map from $X$ to $A_1$. We thus have two morphisms of subobjects of $A_0 \times A_0$: $X \to A_1$ and $A_1 \to X$. So these two morphisms are inverse to each other.

On the other hand, let us suppose that the canonical map $A_1 \to X$ is an isomorphism and let us consider a morphism $(g,g') : A' \to A \times A$ so that $fg \sim fg'$. Then, $(g_0,g'_0)$ factorises through $X$ and thus through $A_1$. So $g \sim g'$. Hence $[f]$ is monic.
\end{proof}

\subsection{Stacking equivalence groupoids}

\begin{definition}
Let $\categ D$ be a full subcategory of $\categ C$. Then, we define $\mathcal{E}_h(\categ D, \categ C)$ as the full subcategory of $\mathcal{E}_h( \categ C)$ spanned by equivalence groupoids $A$ so that $A_0$ belongs to the full subcategory $\categ D$.
\end{definition}

The category $\categ C$ is canonically a full subcategory of $\mathcal{E}_h(\categ C)$. We are interested in $\mathcal{E}_h(\categ C, \mathcal{E}_h(\categ C))$.

An object $A \in \mathcal{E}_h(\categ C, \mathcal{E}_h(\categ C))$ is in particular defined as an object $A_0 \in \categ C$ and an equivalence groupoid $A_1$ which contains two objects $A_{1,0} \in \categ C$ and $A_{1,1} \in \categ C$. But the fact that the map $A_1 \to A_0 \times A_0$ is monic in $\mathcal E_h(\categ C)$ implies that the map
$$
A_{1,1} \to A_{1,0} \times_{A_0^2} A_0^2 \times_{A_0^2} A_{1,0} = A_{1,0}\times_{A_0^2} A_{1,0}
$$
is an isomorphism by Proposition \ref{cormonicmap}.

\begin{proposition}
One has a canonical functor 
$$
d : \mathcal{E}^w(\categ C)\to
\mathcal{E}_h(\categ C, \mathcal{E}_h(\categ C))
$$
that is essentially surjective.
\end{proposition}

\begin{proof}
Let $A \in \mathcal{E}^w_h(\categ C)$ be a weak equivalence groupoid.
Then, let us denote
\begin{align*}
    dA_0 &= A_0;
    \\
    dA_{1, 0} &= A_1;
    \\
    dA_{1, 1} &= A_1  \times_{A_0^2} A_1 .
\end{align*}
Then, the pair $(dA_{1, 0}, dA_{1, 1})$ has the canonical structure of an equivalence groupoid given by the maps
\begin{align*}
& dA_{1,1} = A_1  \times_{A_0^2} A_1 \to A_{1}^2 = dA_{1, 0} \times dA_{1, 0}
\\ 
& dA_{1, 0} = A_1 \xrightarrow{\text{diagonal}} A_1  \times_{A_0^2} A_1 = dA_{1,1}.
\end{align*}
We denote $dA_1$ this weak equivalence groupoid.
Moreover, the pair $(dA_0 , dA_1)$ has the canonical structure of a 1-coskeletal simplicial object in $\mathcal{E}_h(\categ C)$. Indeed, the face morphism $dA_1 \to dA_0 \times dA_0$ is given by the maps
\begin{align*}
    &dA_{0,1} = A_1 \to A_0 \times A_0 = dA_0 \times dA_0;
    \\
    &dA_{1,1} = A_1  \times_{A_0^2} A_1 \to A_0 \times A_0 = dA_0 \times dA_0;
\end{align*}
and the degeneracy is the map
$$
dA_0 = A_0 \to A_1 \to dA_1.
$$
Proposition \ref{cormonicmap} tells us that the map $dA_1 \to dA_0 \times dA_0$ is monic, and the Kan structure on $A$ gives a Kan structure on $dA$. Thus $dA$ is an equivalence groupoid in equivalence groupoids. One can then notice that the construction $A \mapsto dA$ is natural and defines a functor from $\mathcal{E}^w(\categ C)$ to 
$\mathcal{E}_h(\categ C, \mathcal{E}_h(\categ C))$.

Now, let $B$ be an object of $\mathcal{E}_h(\categ C, \mathcal{E}_h(\categ C))$. Let us denote
$$
A_0 = B_0; A_1 = B_{1,0} .
$$
Let us choose a map $s: A_0 \to A_1$ that represents the degeneracy map $B_0 \to B_1$. The face map $B_1 \to B_0 \times B_0$ gives us a map  $A_1 \to A_0 \times A_0$ whom $s$ is a common section. Thus $A$ is a 1-cokeletal simplicial object in $\categ C$. The Kan structure on $B$ gives a Kan structure on $A$. Thus $A$ is a weak equivalence groupoid. Moreover, one has a canonical isomorphism in $\mathcal{E}_h(\categ C, \mathcal{E}_h(\categ C))$
$$
B \simeq dA .
$$
\end{proof}

\begin{definition}
Let $\mathcal{E}_h^{(2)}(\categ C)$ be the category, called the homotopy category of equivalence 2-groupoids, whose
\begin{enumerate}
    \item objects are those of $\mathcal E^w(\categ C)$, that is weak equivalence groupoids also called equivalence 2-groupoids in $\categ C$;
    \item morphisms from $A$ to $B$ are those of $\mathcal{E}_h(\categ C, \mathcal{E}_h(\categ C))$, using the previous functor from $\mathcal{E}_h(\categ C)$ to
$\mathcal{E}_h(\categ C, \mathcal{E}_h(\categ C))$.
\end{enumerate}
In particular this previous functor factorises as
$$
\mathcal{E}^w(\categ C)\to 
\mathcal{E}_h^{(2)}(\categ C) \to
\mathcal{E}_h(\categ C, \mathcal{E}_h(\categ C)).
$$
The first component is the identity on objects and the second one is an equivalence.
\end{definition}

\begin{remark}
We work with the category $\mathcal{E}_h^{(2)}(\categ C)$ instead of the equivalent category 
$\mathcal{E}_h(\categ C, \mathcal{E}_h(\categ C))$ because its objects are easier to handle.
\end{remark}

Let us give another presentation of the the homotopy category of equivalence 2-groupoids $\mathcal{E}_h^{(2)}(\categ C)$.

\begin{definition}
Let $\mathcal{E}^{(2)}(\categ C)$ be the category of equivalence 2-groupoids whose
\begin{enumerate}
    \item objects are weak equivalence groupoids also called equivalence 2-groupoids;
    \item morphisms of equivalence 2-groupoids from $A$ to $B$ are the data of maps $f_0: A_0 \to B_0$ and $f_1: A_1 \to B_1$ so that the following diagram commutes
    $$
    \begin{tikzcd}
        A_1
        \ar[r,"f_1"] \ar[d]
        & B_1
        \ar[d]
        \\
        A_0 \times A_0
        \ar[r, "f_0 \times f_0"']
        & B_0 \times B_0 .
    \end{tikzcd}
    $$
\end{enumerate}
\end{definition}

\begin{remark}
An equivalence 2-groupoid is the same thing as a weak equivalence groupoid. However, a morphism of weak equivalence groupoid is in particular a morphism of equivalence 2-groupoids but there are
morphisms of equivalence 2-groupoids that are not morphisms of weak equivalence groupoids.
\end{remark}

\begin{definition}
Two morphisms of equivalence 2-groupoids $f,g: A \to B$ are called homotopic if the map
$$
(f_0, g_0): A_0 \to B_0 \times B_0
$$
factorises through $B_1$. This is an equivalence  relation on $\hom_{\mathcal{E}^{(2)}(\categ C)}(A,B)$. Moreover, for any pair of morphisms $f', g': B \to B'$:
$$
f \sim g \text{ and } f' \sim g' \implies f' \circ f \sim g' \circ g . 
$$
\end{definition}

\begin{proposition}
The category $\mathcal{E}_h^{(2)}(\categ C)$ is obtained from $\mathcal{E}^{(2)}(\categ C)$ by merging homotopic morphisms. That explains why we called  $\mathcal{E}_h^{(2)}(\categ C)$ the homotopy category of equivalence 2-groupoids.
\end{proposition}

\begin{proof}
Straightforward.
\end{proof}

\begin{corollary}\label{corepicmap}
Let $f: A \to B$ be a morphism of equivalence 2-groupoids. If $f_0$ is an isomorphism, then the induced morphism in $\mathcal E_h^{(2)}(\categ C)$ is epic.
\end{corollary}

\begin{proof}
Let us consider two maps $g,g' : B \to Y$. If $ gf$ is homotopic to $g'f$, then the map
$$
A_0 \xrightarrow{f_0} B_0 \xrightarrow{(g_0,g'_0)} Y_0 \times Y_0
$$
factorises through $Y_1$. But, since the map $f_0$ is an isomorphism, then the map $(g_0,g'_0)$ itself factorises through $Y_1$. Hence $g$ is homotopic to $g'$.
\end{proof}

\subsection{A remark on strict units}

One can represent some notion of a monoid up to homotopy in a lex category $\categ C$ using simplicial objects in $\categ C$ (for instance weak equivalence groupoids). However, such homotopical monoids and their homotopical morphisms encoded by simplicial objects and morphisms of simplicial objects are always strict unital. For instance, in a simplicial object $X$, the operation of composition with the unit on the left may be encoded as the composition
$$
X_1 \xrightarrow{\sigma_0} X_2 \xrightarrow{\delta_1} X_1
$$
which is strictly the identity of $X_1$. Similarly, a morphism of simplicial objects $f: X \to Y$ commutes strictly with respect to units in the sense that the following square is strictly commutative
$$
\begin{tikzcd}
    X_1
    \ar[r]
    & Y_1
    \\
    X_0
    \ar[u, "\sigma"] \ar[r]
    & Y_0 .
    \ar[u, "\sigma"]
\end{tikzcd}
$$

This is the reason why we use $\mathcal{E}_h^{(2)}(\categ C)$ instead of $\mathcal{E}_h^w(\categ C)$ to describe equivalence groupoids up to homotopy in $\categ C$. Indeed, for any morphism of equivalence 2-groupoids $f:X \to Y$, the diagram
$$
\begin{tikzcd}
    X_1
    \ar[r]
    & Y_1
    \\
    X_0
    \ar[u, "\sigma"] \ar[r]
    & Y_0 .
    \ar[u, "\sigma"]
\end{tikzcd}
$$
commutes only up to homotopy in the sense that the map
from $X_0$ to $Y_1 \times Y_1$ factorises through $Y_{1,1} = Y_1 \times_{Y_0^2} Y_1$.

\subsection{Finite limits in equivalence 2-groupoids}

\begin{definition}
Let us consider a cospan diagram $A \to D \leftarrow B$ in $\mathcal{E}^{(2)}(\categ C)$. Then, we denote $A \times^{h}_D B$ the 1-coskeletal simplicial object in $\categ C$ so that
\begin{align*}
(A \times^{h}_D B)_0 =& A_0 \times_{D_0} D_1 \times_{D_0} B_0
\\
(A \times^{h}_D B)_1 =& A_1 \times_{D_0 \times D_0} (D_1 \times D_1) \times_{D_0 \times D_0} B_1,
\end{align*}
where the second formula actually denotes the limit in $\categ C$ os the following diagram
$$
\begin{tikzcd}
    D_0
    & A_0
    \ar[l]
    & A_1
    \ar[l, "d_0"'] \ar[r, "d_1"]
    & A_0
    \ar[r]
    & D_0
    \\
    D_1
    \ar[u, "d_0"] \ar[d, "d_1"']
    &&&& D_1
    \ar[u, "d_0"'] \ar[d, "d_1"]
    \\
    D_0
    & B_0
    \ar[l]
    & B_1
    \ar[l, "d_0"'] \ar[r, "d_1"]
    & B_0
    \ar[r]
    & D_0 .
\end{tikzcd}
$$
The face map is given by the face maps of $A$ and $B$ and the degeneracy map is given by those of $A$ and $B$ and the diagonal of $D$.

A straightforward checking shows that $A \times^{h}_D B$ is an equivalence 2-groupoid, that is it satisfies the Kan condition.
\end{definition}

\begin{proposition}
The category $\mathcal{E}_h^{(2)}(\categ C)$ admits finite limits.
\end{proposition}

\begin{proof}
It has a final object which is the final simplicial object in $\categ C$.
Moreover, for any cospan diagram $A \to D \leftarrow B$ of equivalence 2-groupoids, one can check that its pullback in $\mathcal{E}_h^{(2)}(\categ C)$ is given by $A \times^{h}_D B$.
\end{proof}

\begin{proposition}
The full subcategory of $\mathcal{E}^{(2)}(\categ C)$ (resp. $\mathcal{E}^{(2)}_h(\categ C)$) spanned by equivalence groupoids identifies canonically with $\mathcal{E}(\categ C)$ (resp. $\mathcal{E}_h(\categ C)$). Moreover, the canonical fully faithful embedding
$$
\mathcal{E}_h(\categ C) \to \mathcal{E}_h^{(2)}(\categ C)
$$
is lex. Subsequently, the canonical fully faithful embedding
$$
\categ C \to \mathcal{E}_h(\categ C) \to \mathcal{E}_h^{(2)}(\categ C)
$$
is also lex.
\end{proposition}

\begin{proof}
The first part of the proposition follows from a straighforward checking. The lex property follows from the fact that
$$
A \times^\Delta_D B = A \times^h_D B
$$
for any cospan diagram $A \to D \leftarrow B$ of equivalence groupoids.
\end{proof}

\subsection{Quotients in equivalence 2-groupoids}

Let us consider a diagram
$$
D : \Delta_{\leq 1}^\op \to \mathcal{E}^{(2)}(\categ C)
$$
that becomes an equivalence 2-groupoid in the homotopy category $\mathcal{E}_h^{(2)}(\categ C)$.
Let us denote
\begin{align*}
    D_0 &= A;
    \\
    D_1 &= B.
\end{align*}
Then, let us denote
\begin{align*}
    cD_0 &= A_0
    \\
    cD_1 &= A_1 \times_{A_0} B_0 \times_{A_0} A_1
\end{align*}
We have a face map
$$
cD_1 = A_1 \times_{A_0} B_0 \times_{A_0} A_1 \to A_1 \times A_1 \xrightarrow{d_0^A \times d_1^A} A_0 \times A_0 = cD_0 \times cD_0.
$$
Moreover, the composite morphism
$$
A \to B \to A \times A
$$
is homotopic to the diagonal of $A$. Such an homotopy gives a degeneracy map
$$
cD_0 = A_0 \to A_1 \times_{A_0} B_0 \times_{A_0} A_1 = cD_1
$$
that is a section of the two previous face maps from $cD_1$ to $cD_0$.

\begin{definition}
Let $cD$ be the 1-coskeletal simplicial object in $\categ C$ whose structural objects are $cD_0$ and $cD_1$ and whose structural maps are described in the paragraph just above.
\end{definition}

\begin{lemma}
Since, the image of $D$ in $\mathcal{E}_h^{(2)}(\categ C)$ has the Kan property, then $cD$ is an equivalence 2-groupoid.
\end{lemma}

\begin{proof}
It amounts to show that the map $cD(\partial \Delta[2]) \to cD(\Lambda^i[2])$ has a section for $i=0,1,2$. Let us deal with the case $i=1$.
The fact that $D$ is Kan in $\mathcal{E}_h^{(2)}(\categ C)$ gives us a map
$$
B \times^h_A B \to B
$$
so that the following diagram commutes (up to homotopy)
$$
\begin{tikzcd}
    B \times^h_A B 
    \ar[r] \ar[rd]
    &B
    \ar[d]
    \\
    & A  \times A .
\end{tikzcd}
$$
The related level-0 map together with homotopies that make this diagram commute give us a morphism in $\categ C$
$$
B_0 \times_{A_0} A_1 \times_{A_0} B_0 \to A_1 \times_{A_0} B_0 \times_{A_0} A_1
$$
above $A_0 \times A_0$. Then, using also the Kan property on $A$ we get a map $f$ above $A_0 \times A_0$:
$$
\begin{tikzcd}
{cD(\Lambda^1[2])}
\ar[d, equal]
\\
A_1 \times_{A_0} B_0 \times_{A_0} A_1 \times_{A_0} A_1 \times_{A_0} B_0 \times_{A_0} A_1
\ar[d]
\\
A_1 \times_{A_0} B_0 \times_{A_0} A_1  \times_{A_0} B_0 \times_{A_0} A_1
\ar[d]
\\
A_1 \times_{A_0} A_1 \times_{A_0} B_0 \times_{A_0} A_1 \times_{A_0} A_1
\ar[d]
\\
A_1 \times_{A_0} B_0 \times_{A_0} A_1 
\ar[d, equal]
\\
cD_1
\end{tikzcd}
$$
This gives the expected section map
$$
cD(\Lambda^1[2]) \xrightarrow{(f, \id)} cD_1 \times_{cD_0 \times cD_0} cD(\Lambda^1[2]) = cD(\partial \Delta[2])
$$
The cases $i=0$ and $i=2$ may be proven using the same arguments.
\end{proof}

\begin{proposition}\label{propositioncolimitse}
There exists a map $A \to cD$ in $\mathcal{E}^{(2)}(\categ C)$ that makes $cD$ the coequaliser in $\mathcal{E}_h^{(2)}(\categ C)$ of
the pair of maps $B \rightrightarrows A$.
\end{proposition}

\begin{proof}
Let us describe such a morphism $f: A \to cD$. On the one hand, one can take $f_0 = \id_{A_0}$. On the other hand the level 1 map $f_1$ is the composition
$$
\begin{tikzcd}
    A_1
    \ar[d, equal]
    \\
    cD_0 \times_{A_0} A_1 \times_{A_0} cD_0
    \ar[d]
    \\
    cD_1 \times_{A_0} A_1 \times_{A_0} cD_1
    \ar[d, equal]
    \\
    A_1 \times_{A_0} B_0 \times_{A_0} A_1 \times_{A_0} A_1 \times_{A_0} A_1 \times_{A_0} B_0 \times_{A_0} A_1
    \ar[d, "{\text{using the Kan structure on }A_1}"]
    \\
    A_1 \times_{A_0} B_0 \times_{A_0} A_1
    \times_{A_0} A_1
    \times_{A_0} B_0 \times_{A_0} A_1
    \ar[d, equal]
    \\
    cD_1 \times_{A_0} cD_1
    \ar[d, "{\text{using the Kan structure on }cD_1}"]
    \\
    cD_1 .
\end{tikzcd}
$$
It is clear that the two composite maps $B \rightrightarrows A \to cD$ are homotopic since the two maps $B_0 \to A_0$ factorise in a trivial way through $cD_1$.

Let $U$ be a weak equivalence groupoid. The function
$$
\hom_{\mathcal{E}^{(2)}_h(\categ C)}(cD, U) \to \mathrm{equaliser}\left(\hom_{\mathcal{E}^{(2)}_h(\categ C)}(A, U) \rightrightarrows \hom_{\mathcal{E}^{(2)}_h(\categ C)}(B, U)\right)
$$
is monic since the map $A \to cD$ is epic (Corollary \ref{corepicmap}). Hence, it is an isomorphism if and only if it has a section. Let us build such a section. Let us consider a map of equivalence 2-groupoids $g : A \to U$ so that the two induced maps from $B$ to $U$ are homotopic. Such an homotopy gives a map $h : B_0 \to U_1$ so that the following square diagram commutes
$$
\begin{tikzcd}
    B_0
    \ar[r, "h"] \ar[d]
    & U_1
    \ar[d]
    \\
    A_0 \times A_0
    \ar[r, swap, "g_0 \times g_0"]
    & U_0 \times U_0 .
\end{tikzcd}
$$
We can build another map
$$
\begin{tikzcd}
cD_1 
\ar[r, equal]
& A_1 \times_{A_0} B_0 \times_{A_0} A_1
\ar[d,"{g_1 \times h \times g_1}"]
\\
&U_1 \times_{U_0} U_1 \times_{U_0} U_1    
\ar[d, "\text{Kan structure of }U"]
\\
&U_1
\end{tikzcd}
$$
above $U_0 \times U_0$. Together with the map $g_0: cD_0 =A_0 \to U_0$, this gives us a morphism $cD \to U$ in $\mathcal{E}^{(2)}(\categ C)$. Moreover, the composition $A \to cD \to U$ is homotopic to $g$ since both morphisms are represented by the same level 0 map $g_0 : A_0 \to U_0$. We thus have built a section of the function
$$
\hom_{\mathcal{E}^{(2)}_h(\categ C)}(cD, U) \to \mathrm{equaliser}\left(\hom_{\mathcal{E}^{(2)}_h(\categ C)}(A, U) \rightrightarrows \hom_{\mathcal{E}^{(2)}_h(\categ C)}(B, U)\right)
$$
which is subsequently an isomorphism. Hence, $cD$ is the coequaliser in $\mathcal{E}^{(2)}_h(\categ C)$ of the pair of maps $B \rightrightarrows A$.
\end{proof}

\begin{corollary}\label{elma}
Let $X$ be an equivalence 2-groupoid. Then $X$ is the coequaliser in the category $\mathcal{E}_h^{(2)}(\categ C)$ of the diagram
$$
X_1 \rightrightarrows X_0 .
$$
\end{corollary}

\begin{proof}
This is a direct consequence of Proposition \ref{propositioncolimitse} but it may also be proven using a straightforward checking.
\end{proof}

\begin{corollary}\label{corollaryeffectiveepi}
Let $f: A \to B$ be a morphism of equivalence 2-groupoids. If $f_0$ is an isomorphism, then the image of $f$ the homotopy category $\mathcal{E}_h^{(2)}(\categ C)$ is an effective epimorphism.
\end{corollary}

\begin{proof}
It is straightforward to check that $B$ is the coequaliser through $f$ of the pair of maps 
$$
B_1 \rightrightarrows B_0 \simeq A_0 \to A .
$$
\end{proof}

\begin{corollary}\label{corcompactobject}
Let $X$ be an object of $\categ C$. Then the lex functor
$$
\hom_{\mathcal{E}_h^{(2)}(\categ C)} (X, -)
$$
preserves coequalisers of equivalence 2-groupoids.
\end{corollary}

\begin{proof}
Let $D$ be an equivalence 2-groupoid in $\mathcal{E}_h^{(2)}(\categ C)$ and let us use the notations $A, B, cD$ as above. The composite map
$$
\hom(X, A_0) \to \hom(X, A) \to \mathrm{coeq}(\hom(X, B) \rightrightarrows \hom(X, A)) \to \hom(X, cD) 
$$
makes both $\mathrm{coeq}(\hom(X, B) \rightrightarrows \hom(X, A))$ and $\hom(X, cD)$ appear as the quotient of $\hom(X, A_0)$ by the equivalence relation generated by the two relations
\begin{enumerate}
    \item $f \sim_a g$ if the morphism $(f,g): X \to A_0 \times A_0$ factorises through $A_1$;
    \item $f \sim_b g$ if the morphism $(f,g): X \to A_0 \times A_0$ factorises through $B_0$.
\end{enumerate}
Hence the map $\mathrm{coeq}(\hom(X, B) \rightrightarrows \hom(X, A)) \to \hom(X, cD)$ is an isomorphism.
\end{proof}

\subsection{Quotients as left adjoint functors}

\begin{definition}
Let $A$ be an equivalence 2-groupoid in $\categ C$. If it exists, we denote $A_0/A_1$ the quotient of $A$, that is the coequaliser in $\categ C$ of the pair of maps
$$
A_1 \rightrightarrows A_0.
$$
\end{definition}

\begin{proposition}\label{propquotientleftadj}
The three following assertions are equivalent.
\begin{enumerate}
    \item the embedding $\categ C \to \mathcal{E}^{(2)}(\categ C)$ has a left adjoint;
    \item the embedding $\categ C \to \mathcal{E}_h^{(2)}(\categ C)$ has a left adjoint;
    \item the category $\categ C$ admits quotients of equivalence 2-groupoids.
\end{enumerate}
\end{proposition}

\begin{proof}
By definition of colimits, (1) is equivalent to (3).

Let us suppose (1). Then, the "quotient" left adjoint functor $A \mapsto A_0/A_1$ from $\mathcal{E}^{(2)}(\categ C)$ to $\categ C$ factorises through $\mathcal{E}_h^{(2)}(\categ C)$. The resulting functor from $\mathcal{E}_h^{(2)}(\categ C)$ to $\categ C$ is left adjoint to the embedding because
the map
$$
\hom_{\mathcal{E}^{(2)}(\categ C)} (A, X) \to 
\hom_{\mathcal{E}_h^{(2)}(\categ C)} (A, X)
$$
is bijective for any equivalence 2-groupoid $A$ and any element $X$ in $\categ C$.

If we suppose (2), then the image through the left adjoint functor from $\mathcal{E}_h^{(2)}(\categ C)$ to $\categ C$ of an equivalence 2-groupoid is its quotient. This shows (3).
\end{proof}

\begin{corollary}
The three following assertions are equivalent.
\begin{enumerate}
    \item the embedding $\categ C \to \mathcal{E}(\categ C)$ has a left adjoint;
    \item the embedding $\categ C \to \mathcal{E}_h(\categ C)$ has a left adjoint;
    \item the category $\categ C$ admits quotients of equivalence groupoids.
\end{enumerate}
\end{corollary}

\begin{proof}
This follows from the same arguments as those used to prove Proposition \ref{propquotientleftadj}.
\end{proof}

\subsection{Changing the ground category}

\begin{lemma}\label{kanext}
Let $A$ be an element of $\mathcal{E}_h^{(2)}(\categ C)$. Then the functor
\begin{align*}
\Delta_{\leq 1}^\op & \to \categ C/A
\\
i & \mapsto (A_i \to A)
\end{align*}
is cofinal.
\end{lemma}

\begin{proof}
Let us consider a map $f : U \to A$ where $U \in \categ C$. Then, one can show that the category $f/(\Delta_{\leq 1}^\op)$ is nonempty and connected. This proves the result.
\end{proof}

\begin{lemma}\label{kanexthree}
Let us consider a lex functor $F : \categ C \to \categ D$ and let us suppose that $\categ D$ has quotients of equivalence 2-groupoids. Then, the functor
\begin{align*}
    \overline F : \mathcal{E}_h^{(2)}(\categ C) & \to \categ D
    \\
    (A_0, A_1) & \mapsto F(A_0)/F(A_1)
\end{align*}
is the pointwise left Kan extension of $F$
along the inclusion $\categ C \to \mathcal{E}_h(\categ C)$.
The same phenomenon happens if we replace $\mathcal{E}_h^{(2)}$ by $\mathcal{E}_h$ and equivalence 2-groupoids by equivalence groupoids. 
\end{lemma}

\begin{proof}
This is a direct consequenc of Lemma \ref{kanext}.
\end{proof}

\begin{corollary}\label{kanexttwo}
Let us consider a lex functor $F : \categ C \to \categ D$. Then, the functor
\begin{align*}
    \mathcal{E}_h^{(2)}(F) : \mathcal{E}_h^{(2)}(\categ C) & \to \mathcal{E}_h^{(2)}(\categ D)
    \\
    (A_0, A_1) & \mapsto (F(A_0), F(A_1))
\end{align*}
is the pointwise left Kan extension of the composite functor
$$
\categ C \to \categ D \to \mathcal{E}_h^{(2)}(\categ D)
$$
along the inclusion $\categ C \to \mathcal{E}_h(\categ C)$.
The same phenomenon happens if we replace $\mathcal{E}_h^{(2)}$ by $\mathcal{E}_h$. 
\end{corollary}

\begin{proof}
This follows from Lemma \ref{kanexthree} combined with Corollary \ref{elma}.
\end{proof}

\begin{proposition}
The construction $\categ C \in \mathfrak{Lex} \mapsto \mathcal{E}_h^{(2)}(\categ C)$ and the canonical fully faithful embeddings $\categ C \to \mathcal{E}_h^{(2)}(\categ C)$ make $\mathcal{E}_h^{(2)}$ a pointwise KZ-machine on the 2-category $\mathfrak{Lex}$.
\end{proposition}

\begin{proof}
This is a consequence of Lemma \ref{kanexthree} and of Corollary \ref{kanexttwo}, and of straightforward checkings.
\end{proof}

\begin{proposition}
The construction $\categ C \in \mathfrak{Lex} \mapsto \mathcal{E}_h(\categ C)$ and the canonical fully faithful embeddings $\categ C \to \mathcal{E}_h(\categ C)$ make $\mathcal{E}_h$ a pointwise KZ-machine on the 2-category $\mathfrak{Lex}$. Moreover, $\mathcal{E}_h$ is simpler than $\mathcal{E}_h^{(2)}$ through the canonical inclusion $\mathcal{E}_h(\categ C) \to
\mathcal{E}_h^{(2)}(\categ C)$.
\end{proposition}

\begin{proof}
Again, the fact that $\mathcal{E}_h$ is a pointwise KZ-machine is a consequence of Lemma \ref{kanexthree} and of Corollary \ref{kanexttwo}, and of straightforward checkings. The fact that it is simpler that $\mathcal{E}_h^{(2)}$ follows from Lemma \ref{kanext}.
\end{proof}

\begin{proposition}\label{propcommdiagr}
Let $f: \categ C \to \categ D$ be a lex functor between $\mathcal{E}_h^{(2)}$-cocomplete lex categories (in particular, they have quotients of equivalence 2-groupoids). The following assertions are equivalent
\begin{enumerate}
    \item $f$ is $\mathcal{E}_h^{(2)}$-linear;
    \item $f$ preserves quotients of equivalence 2-groupoids.
\end{enumerate}
Similarly, if $f: \categ C \to \categ D$ is a lex functor between $\mathcal{E}_h$-cocomplete lex categories (in particular, they have quotients of equivalence groupoids), the following assertions are equivalent
\begin{enumerate}
    \item $f$ is $\mathcal{E}_h$-linear;
    \item $f$ preserves quotients of equivalence groupoids.
\end{enumerate}
\end{proposition}

\begin{proof}
Knowing that the KZ-machines $\mathcal{E}_h^{(2)}$ and $\mathcal{E}_h$ are pointwise, this is a consequence of Lemma \ref{kanext}.
\end{proof}

\subsection{The KZ-machine of equivalence groupoids}

\begin{proposition}\label{propgiraudax}
The following assertions are equivalent:
\begin{enumerate}
    \item $\categ C$ is $\mathcal{E}_h$-cocomplete in $\mathfrak{Lex}$;
    \item equivalence groupoids in $\categ C$ are effective and coequalisers of equivalence groupoids are universal, in the sense that for any equivalence groupoid $X$ and any cospan diagram
    $X \to Z \leftarrow Y$ where $Y,Z \in \categ C$, the map
    $$
    \mathrm{coeq}\left(X_1 \times_Z Y \rightrightarrows X_0 \times_Z Y \right) \to (X_0/X_1) \times_Z Y
    $$
    is invertible;
    \item equivalence groupoids in $\categ C$ are effective and effective epimorphisms are stable through pullbacks.
\end{enumerate}
\end{proposition}

\begin{proof}
In any of these cases, $\categ C$ is supposed to have coequalisers of equivalence groupoids. Hence the canonical embedding functor $\categ C \to \mathcal{E}_h(\categ C)$ has a left adjoint that we denote $q$ and that sends an equivalence groupoid $X$ to its quotient $qX = X_0/X_1$. Moreover, $q$ preserves the final object.

Let us suppose (1). One can notice that for any equivalence groupoid $X$, one has
$$
X_0 \times_X^h X_0 = X_1
$$
Hence,
$$
X_1 = q (X_1) = q (X_0 \times_X^h X_0)
= q  (X_0) \times_{qX} q(X_0)
= X_0 \times_{qX} X_0,
$$
which means that the equivalence groupoid $X$ is effective.

It is then straightforward to show that (2) is a consequence of (1) and (3) is a consequence of (2).

Let us assume (3) and let us prove (2). Let us consider a cospan diagram $X \to Y \leftarrow Z$ in $\mathcal{E}_h(\categ C)$, so that $Y$ and $Z$ belong to $\categ C$. Since effective epimorphisms are stable through pullbacks, then the map
$$
X_0 \times_{Z} Y \to qX \times_Z Y
$$
is an effective epimorphism. Its related equivalence relation on $X_0 \times_{Z} Y$ is
$$
(X_0 \times_{Z} Y) \times_{qX \times_{Z} Y} (X_0 \times_{Z} Y) = (X_0 \times_{qX} X_0) \times_Z Y
$$
which is $X_1 \times_Z Y$ since equivalence groupoids are effective. Thus, $qX \times_Z Y$ is the coequaliser of the maps
$$
X_1 \times_Z Y \rightrightarrows X_0 \times_Z Y.
$$
This shows (2).

Finally, let us assume (2) and let us prove (1). Let us consider a cospan diagram $X \to Y \leftarrow Z$ in $\mathcal{E}_h(\categ C)$. Since the equivalence groupoid $Z$ is effective, then by Lemma \ref{lemmaeffective}, the map
$$
X \times^h_Z Y \to X \times_{qZ}^h Y
$$
is an isomorphism. Moreover, since coequalisers of equivalence groupoids are universal and since the category $\Delta_{\leq 1}^\op$ is sifted, the map
$$
q(X \times_{qZ}^h Y) \to qX \times_{qZ} qY
$$
is the following composite isomorphism
\begin{align*}
    q(X \times_{qZ}^h Y)
    &=\varinjlim_{k \in \Delta_{\leq 1}^\op}\left( X \times_{qZ}^h Y\right)_k
    \\
    &= \varinjlim_{k \in \Delta_{\leq 1}^\op} X_k \times_{qZ} Y_k
    \\
    & \simeq 
    \varinjlim_{(i,j) \in \Delta_{\leq 1}^\op \times \Delta_{\leq 1}^\op}X_i \times_{qZ} Y_j
    \\
    &= \varinjlim_{i \in \Delta_{\leq 1}^\op}
    \varinjlim_{j \in \Delta_{\leq 1}^\op}
    X_i \times_{qZ} Y_j
    \\
    &= \varinjlim_{i \in \Delta_{\leq 1}^\op}
    X_i \times_{qZ} qY
    \\
    &= qX \times_{qZ} qY.
\end{align*}
So $q$ is lex. This shows (1).
\end{proof}

\begin{lemma}\label{lemmaintermedefunctor}
Let us suppose that the lex category $\categ C$ is $\mathcal{E}_h$-cocomplete. Then the composite functor
$$
e : \mathcal{E}_h^{(2)}(\categ C) \xrightarrow{\sim} \mathcal{E}_h(\categ C, \mathcal{E}_h(\categ C))  \hookrightarrow \mathcal{E}_h\mathcal{E}_h(\categ C) \xrightarrow{\mathcal{E}_h(m_{\mathcal{E}_h}(\categ C))} \mathcal{E}_h(\categ C)
$$
is left adjoint to the canonical fully faithful embedding $\mathcal{E}_h(\categ C)\hookrightarrow \mathcal{E}_h^{(2)}(\categ C)$.
\end{lemma}

\begin{proof}
This functor sends an equivalence 2-groupoid $A$ to the equivalence groupoid $eA$
so that
\begin{align*}
    eA_0 &= A_0
    \\
    eA_1 &= A_1/ A_{1,1}, \text{ where }A_{1,1} = A_1 \times_{A_0^2} A_1.
\end{align*}
The unit of the adjunction is the canonical natural morphism from $A$ to $eA$ induced by the equality $eA_0 = A_0$ and the quotient map $A_1 \to A_1 /A_{1,1}$. If $A$ is an equivalence groupoid, this map $A \to eA$ is an isomorphism. The inverse map is the counit of the adjunction.
\end{proof}

\begin{proposition}
\label{lemmaintermediateleftadjoint}
If the lex category $\categ C$ is $\mathcal{E}_h$-cocomplete, then it is $\mathcal{E}_h^{(2)}$-cocomplete.
\end{proposition}

\begin{proof}
Let us denote $q$ the quotient functor
$$
q = m_{\mathcal{E}_h}(\categ C) : \mathcal{E}_h(\categ C) \to \categ C.
$$
From Lemma \ref{lemmaintermedefunctor} and from the fact that $q$ is left adjoint to the fully faithful embedding of $\categ C$ into $\mathcal{E}_h(\categ C)$, we know that the composite functor
$$
qe : \mathcal{E}_h^{(2)}(\categ C) \xrightarrow{q} \mathcal{E}_h(\categ C) \xrightarrow{e} \categ C
$$
is left adjoint to the embedding of $\categ C$ into $\mathcal{E}_h^{(2)}(\categ C)$. The counit of the adjunction is invertible since the right adjoint is fully faithful.

It remains to show that $qe$ is lex.
It is clear that it preserves the final object. Let us show that it preserves pullbacks. Let $A \to D \leftarrow B$ be a cospan diagram in $\mathcal{E}^{(2)}(\categ C)$.
We need to show that the composite map
$$
qe (A \times^{h}_D B) \xrightarrow{f} q(eA \times^{h}_{eD} eB) \xrightarrow{f'} qeA \times_{qeD} qeB
$$
is invertible. Since $q$ is lex, $f'$ is invertible. So, let us prove that $f$ is also an isomorphism. Let us denote
\begin{enumerate}
    \item $X = A \times^h_D B$
    \item $Y = e(A) \times^h_{e(D)} e(B) $
    \item $Z = A_0 \times_{D_0}^h {dD_1} \times_{D_0}^h B_0$ where $dD$ is the equivalence groupoid so that
    \begin{align*}
        dD_{1,0} &= D_1;
        \\
        dD_{1,1} &=  D_1 \times_{D_0^2} D_1.
    \end{align*}
\end{enumerate}
The map $f$ fits in the following commutative diagram in $\mathcal{E}^{(2)}(\categ C)$:
$$
\begin{tikzcd}
    Z 
    \ar[r, "\tilde g"]
    \ar[d]
    &X
    \ar[r, "\tilde f"] \ar[d]
    & Y
    \ar[d]
    \\
    Y_0
    \ar[r, "g"']
    &qeX
    \ar[r, "f"']
    & qY
\end{tikzcd}
$$
where the vertical maps are projections onto quotients. The quotient of 
$Z$ is $Y_0$ since $q$ is lex 
$$
Y_0 = A_0 \times_{D_0} eD_1 \times_{D_0} B_0 =A_0 \times_{D_0} qdD_1 \times_{D_0} B_0
= q(A_0 \times_{D_0}^h dD_1 \times_{D_0}^h B_0).
$$
We denote $p_Z$ the composite map $Z_0 \to Z \to Y_0$ and we define similarly $p_X: X_0 \to X \to qeX$ and $p_Y : Y_0 \to Y \to qY$. One can notice that the composition $Y_0 \to Y \to qY$ is equal to the horizontal map $f \circ g$ and that the map
$$
\tilde f_1:A_1 \times_{D_0^2} D_1^2 \times_{D_0^2} B_1
\to eA_1 \times_{D_0^2} eD_1^2 \times_{eD_0^2} eB_1
$$
is an effective epimorphism since effective epimorphisms in $\categ C$ are stable through pullbacks (Proposition \ref{propgiraudax}).

The map $\tilde g_0 = Z_0 \to X_0$ is invertible (this is actually the equality of $A_0 \times_{D_0} D_1 \times_{D_0} B_0$). Moreover, the map $p_Z \circ \tilde g_0^{-1} : X_0 \simeq Z_0 \to Y_0$ is equal to $\tilde f_0$. Then, the pair of maps $X_1 \rightrightarrows X_0 \to qeX$ factorises as
$$
X_1 \rightrightarrows X_0 \xrightarrow{\tilde g_0^{-1}} Z_0 \xrightarrow{p_Z} Y_0 \xrightarrow{g} qeX
$$
which rewrites as
$$
X_1 \xrightarrow{\tilde f_1} Y_1 \rightrightarrows Y_0 \xrightarrow{g} qeX.
$$
Since these two composite maps from $X_1$ to $qeX$ are equal and since the map $\tilde f_1$ is an effective epimorphism, the two maps
$$
Y_1 \rightrightarrows Y_0 \to qeX
$$
are equal. This gives us a map $i : qY \to qeX$ so that, by definition of quotients, $g =i \circ p_Y$ and $f\circ i = \id_{qY}$. Finally,
$$
g = i \circ p_Y = i \circ  f\circ i \circ p_Y = i \circ  f\circ g.
$$
Since $g$ is epic (because $p_X = g \circ p_Z \circ \tilde g_0^{-1}$ is epic), then $i \circ f = \id$.
So $i$ is inverse to $f$.
\end{proof}

\begin{lemma}\label{LemmaMoritamor}
Let $f: \categ C \to \categ D$ be a lex functor between $\mathcal{E}_h$-cocomplete categories (in particular they are $\mathcal{E}_h^{(2)}$-cocomplete by Proposition \ref{lemmaintermediateleftadjoint}). If $f$ is $\mathcal{E}_h$-linear, then it is $\mathcal{E}_h^{(2)}$-linear.
\end{lemma}

\begin{proof}
Let $A$ be an equivalence 2-groupoid in $\categ C$. Then the map $f(A_0)/f(A_1) \to f(A_0/A_1)$ is an isomorphism as it is equal to the sequence of canonical maps
$$
\begin{tikzcd}
    f(A_0)/f(A_1)
    \ar[r, equal]
    &f(A_0)/(f(A_1)/f(A_{1,1}))
    \ar[d]
    \\
    & f(A_0)/ f(A_1/A_{1,1})
    \ar[d]
    \\
    &f(A_0/(A_1/A_{1,1}))
    \ar[r, equal]
    &f(A_0/A_1)
\end{tikzcd}
$$
which are all isomorphisms.
\end{proof}

\begin{proposition}
The two KZ-machines $\mathcal{E}_h$ and $\mathcal{E}_h^{(2)}$ on $\mathfrak{Lex}$ are Morita equivalent.
\end{proposition}

\begin{proof}
The KZ-machine $\mathcal{E}_h$ is simpler than $\mathcal{E}_h^{(2)}$, so it is Morita simpler.
Conversely, Proposition \ref{lemmaintermediateleftadjoint} and Lemma \ref{LemmaMoritamor} tell us that $\mathcal{E}_h^{(2)}$ is Morita simpler that $\mathcal{E}_h$. Hence they are Morita equivalent.
\end{proof}

\subsection{The KZ-doctrine of equivalence 2-groupoids}

\begin{theorem}
The KZ-machine of equivalence 2-groupoids $\mathcal{E}^{(2)}_h$ on $\mathfrak{Lex}$ is a KZ-doctrine.
\end{theorem}

\begin{proof}
This is a consequence of the two following lemmas combined with Proposition \ref{propgiraudax}.
\end{proof}

\begin{lemma}
Effective epimorphisms in $\mathcal{E}^{(2)}_h(\categ C)$ are stable through pullbacks.
\end{lemma}

\begin{proof}
Let us consider a morphism of equivalence 2-groupoids $f: A \to B$ that becomes an effective epimorphism in $\mathcal{E}_h^{(2)}(\categ C)$. From Proposition \ref{propositioncolimitse}, we can assume that the level 0 map $f_0 : A_0 \to B_0$ is an isomorphism. Also, let us consider a cospan diagram $B \to E \leftarrow D$. The map
$$
A \times^h_E D \to B \times^h_E D
$$
is represented by the level-0 map
$$
A_0 \times_{E_0} E_1 \times_{E_0} D_0 \simeq B_0 \times_{E_0} E_1 \times_{E_0} D_0
$$
which is an isomorphism.
Hence, by Corollary \ref{corollaryeffectiveepi}, it is an effective epimorphism.
\end{proof}

\begin{lemma}
Equivalence relations in $\mathcal{E}^{(2)}_h(\categ C)$ are effective.
\end{lemma}

\begin{proof}
Let us consider an equivalence groupoid $d_0, d_1:B \rightrightarrows A$ in $\mathcal{E}^{(2)}_h(\categ C)$, let us denote $K$ its colimit as in Proposition \ref{propositioncolimitse}, that is
\begin{align*}
    K_0 &= A_0
    \\
    K_1 &= A_1 \times_{A_0} B_0 \times_{A_0} A_1.
\end{align*}
We need to show that the map $i: B \to A \times^h_{K} A$ is an isomorphism.
We have
\begin{align*}
    (A \times^h_{K} A)_0 &= K_1
    \\
    (A \times^h_{K} A)_1 &= A_1 \times_{A_0^2}(K_1 \times K_1) \times_{A_0^2} A_1.
\end{align*}

The fact that the map $B \to A \times A$ is monic tells us that the morphism of equivalence groupoids in $\categ C$
$$
dB_1 \to dA_1 \times_{A_0^2} (B_0 \times B_0) \times_{A_0^2} dA_1
$$
is an equivalence. The homotopy inverse map may be represented by a level 0 map above $B_0 \times B_0$
$$
f : A_1 \times_{A_0^2} (B_0 \times B_0) \times_{A_0^2} A_1 \to B_1 .
$$
Moreover, the Kan structure on $A$ gives us another map above $B_0 \times B_0$
$$
g : (A \times^h_{K} A)_1 = A_1 \times_{A_0^2}(K_1 \times K_1) \times_{A_0^2} A_1 \to 
A_1 \times_{A_0^2} (B_0 \times B_0) \times_{A_0^2} A_1.
$$
The data of $r_1 = gf$ and of the projection
$$
r_0 : K_1 = A_1 \times_{A_0} B_0 \times_{A_0} A_1 \to B_0
$$
gives a morphism of equivalence 2-groupoids $r : A \times^h_{K} A \to B$. Since $r_0 \circ i_0 = \id$, $r \circ i$ is homotopic to the identity of $B$. Conversely, it is clear that the composition
$$
A \times^h_K A \xrightarrow{r} B  \xrightarrow{i} A \times^h_K A \to A \times A
$$
is homotopic to the canonical inclusion $(\pi_0, \pi_1) :A \times^h_K A \to A \times A$. Since this inclusion is monic in $\mathcal{E}_h^{(2)}(\categ C)$, then $i \circ r$ is homotopic to the identity of $A \times^h_K A$. Thus $r$ is inverse to $i$ in $\mathcal{E}_h^{(2)}(\categ C)$.
\end{proof}


\section{Small presheaves as equivalence 2-groupoids of sums}

In this last section we compose the KZ-doctrines $\mathcal{E}_h^{(2)}$ and $\mathcal{S}$ that act on the 2-category $\mathfrak{Lex}$ of lex categories in order to recover the KZ-doctrine $\Psh$ of small presheaves.

\subsection{The KZ machines of equivalence 2-groupoids and sums}

\begin{proposition}
The KZ-machine $\mathcal{S}$ on $\mathfrak{Lex}$ extends $\mathcal{E}_h$ and $\mathcal{E}_h^{(2)}$. Conversely, $\mathcal{E}_h$ and $\mathcal{E}_h^{(2)}$ extend $\mathcal S$.
\end{proposition}

\begin{proof}
Let $E$ be a KZ-machine on $\mathfrak{Lex}$ which is either $\mathcal{E}_h^{(2)}$ or $\mathcal{E}_h^{(2)}$.

On the one hand, $\mathcal{S}$ extends $E$ because the functor $i_{\mathcal S}(\categ D) :\categ D \to \mathcal{S}(\categ D)$ preserves colimits of connected diagrams that may exist in $\categ D$ for any lex category $\categ D$.

On the other hand, $E$ extends $\mathcal{S}$ because the functor
$$
i_E(\categ D) : \categ D \to E(\categ D)
$$
preserves coproducts that may appear in $\categ D$ for any lex category $\categ D$.
\end{proof}

\begin{corollary}
The pairs of KZ-doctrines 
\begin{enumerate}
    \item $(\mathcal{E}_h, \mathcal S)$;
    \item $(\mathcal{E}_h^{(2)}, \mathcal S)$;
    \item $(\mathcal S, \mathcal{E}_h)$;
    \item $(\mathcal S, \mathcal{E}_h^{(2)})$;
\end{enumerate}
are composable. Moreover, the four resulting KZ-machines $\mathcal{E}_h \mathcal S, \mathcal{E}_h^{(2)} \mathcal S, \mathcal{S} \mathcal{E}_h, \mathcal{S} \mathcal{E}_h^{(2)}$ are all Morita equivalent and
\begin{enumerate}
    \item their algebras are lex categories that are both $\mathcal{S}$-cocomplete and $\mathcal{E}_h$-cocomplete;
    \item their linear morphisms are lex functors that preserves quotients of equivalence groupoids and small coproducts.
\end{enumerate}
\end{corollary}

\begin{definition}\label{definitionpretopos}
An infinitary pretopos is an $\mathcal{E}_h\mathcal S$-cocomplete object in $\mathfrak{Lex}$. In other words, this is a category with small coproducts and finite limits so that
\begin{enumerate}
    \item equivalence groupoids are effective;
    \item effective epimorphisms are stable through pullbacks;
    \item coproducts are disjoint;
    \item coproducts are universal.
\end{enumerate}
\end{definition}

\subsection{An informal description of equivalence 2-groupoids on sums}

For any lex category $\categ C$, $\mathcal{E}_h^{(2)} \mathcal{S} (\categ C)$ is canonically equivalent to the category:
\begin{enumerate}
    \item whose objects are pairs $(O, X)$ of a 1-cokeletal Kan complex (that is a Kan 1-cokeletal simplicial object in sets) $O$ and a functor $X : \left(\Delta_{\leq 1 /O}\right)^\op \to \categ C$ that satisfies an additional Kan condition, that is
    \begin{itemize}
        \itemt for any $o,o' \in O_1$ so that $d_1(o) = d_0(o')= b$, there exists an element $o'' \in O_1$ so that $d_0(o'')= d_0(o)$, $d_1(o'')= d_1(o')$ together with a morphism 
        $$
        X(o) \times_{X(b)} X(o') \to X(o'')
        $$
        so that the following diagram commutes
        $$
        \begin{tikzcd}
            X(o) \times_{X(b)} X(o')
            \ar[r] \ar[d]
            &
            X(o'')
            \ar[d]
            \\
            X(d_0(o)) \times X(d_1(o'))
            \ar[r, equal]
            & X(d_0(o'')) \times X(d_1(o''));
        \end{tikzcd}
        $$
        \itemt for any $o \in O_1$, there exists an element $o' \in O_1$ so that $d_0(o') = d_1(o)$ and $d_1(o') = d_0(o)$ together with a morphism $X(o) \to X(o')$ so that the following diagram commutes
        $$
        \begin{tikzcd}
            X(o)
            \ar[r] \ar[d]
            &
            X(o')
            \ar[d]
            \\
            X(d_0(o)) \times X(d_1(o))
            \ar[r, "\tau"]
            & X(d_0(o')) \times X(d_1(o'));
        \end{tikzcd}
        $$
    \end{itemize}
    \item whose morphisms from $(O, X)$ to $(Q, Y)$ are equivalence classes of
    the data of a morphism of simplicial sets $f : O \to Q$ together with maps
    $$
    f_o : X(o) \to Y(f(o)), \quad o \in O_0 \text{ or }o \in O_1
    $$
    so that for any $o \in O_1$, the following diagram commutes
    $$
    \begin{tikzcd}
        X(o)
        \ar[r] \ar[d]
        & Y(f(o))
        \ar[d]
        \\
        X(d_0(o)) \times X(d_1(o))
        \ar[r]
        & Y(d_0f(o)) \times Y(d_1f(o)) .
    \end{tikzcd}
    $$
    Two such data $(f, (f_o)_{o \in O_0, O_1})$ and $(g, (g_o)_{o \in O_0, O_1})$ are equivalent if for any $o \in O_0$ there exists an element $q \in Q_1$ relating $f(o)$ to $g(o)$ and so that the map
    $$
    (f_o,  g_o) : X(o) \to Y(f(o)) \times Y(g(o))
    $$
    factorises through $Y(q)$.
\end{enumerate}

\subsection{The KZ doctrine of equivalence 2-groupoids and sums}

\begin{lemma}\label{lemmadecomposition}
Let $\categ C$ be a $\mathcal{S}$-cocomplete lex category and let $f : \categ C \to \categ D$ be a $\mathcal{S}$-linear lex functor. Then $\mathcal{E}_h^{(2)}(\categ C)$ is $\mathcal{S}$-cocomplete and $\mathcal{E}_h^{(2)}(f)$ is $\mathcal{S}$-linear.
\end{lemma}

\begin{proof}
Let us consider a small family of object $(X_i)_{i \in I}$ in $\mathcal{E}_h^{(2)}(\categ C)$. The 1-coskeletal simplicial object $Y$ in $\categ C$ so that
\begin{align*}
Y_0 &= \coprod_i X_{i, 0} ;
\\
Y_1 &= \coprod_i X_{i, 1}.
\end{align*}
has the Kan property and is thus an equivalence 2-groupoid. One can check that $Y$ is the coproduct of the objects $(X_i)_i$. Then, the fact that coproducts are universal and disjoint in $\categ C$ imply in a straightforward way that this is also the case in $\mathcal{E}_h^{(2)}(\categ C)$. Finally, the $\mathcal{S}$-linearity of $\mathcal{E}_h^{(2)}(f)$ just follows from the formula of coproducts.
\end{proof}

\begin{theorem}
The KZ-machine $\mathcal{E}_h^{(2)}\mathcal{S}$ is a KZ-doctrine.
\end{theorem}

\begin{proof}
This is a direct consequence of Lemma \ref{lemmadecomposition}.
\end{proof}

\subsection{Presheaves as equivalence 2-groupoids in sums}

\begin{theorem}
The KZ-doctrines $\mathcal{E}_h^{(2)}\mathcal{S}$ and $\Psh$ on $\mathfrak{Lex}$ are equivalent.
\end{theorem}

One can just say that they are Morita equivalent (see \cite[Proposition 2.5]{GarnerLack} and \cite[Corollary 3.3]{Kock}) since two KZ-doctrines are equivalent if and only if they are Morita equivalent (Proposition \ref{propsimpsimpimplieseq}).
We will prove this theorem in a different way.

\begin{proof}
We know from Lemma \ref{lemmasimplerpresheaves} that the KZ-doctrine $\mathcal{E}_h^{(2)}\mathcal{S}$ is simpler than $\Psh$. Then, for any lex category $\categ C$, the canonical functor
$$
\mathcal{E}_h^{(2)}\mathcal{S}(\categ C) \to \Psh(\categ C)
$$
is fully faithful (Lemma \ref{lemmafullyfaith}) and essentially surjective (Lemma \ref{lemmaessentiallysurjective}).
\end{proof}

\begin{lemma}\label{lemmasimplerpresheaves}
The KZ-doctrines $\mathcal{E}_h^{(2)}\mathcal{S}$ and $\mathcal S$ are simpler than $\Psh$. Moreoverr, the following square diagram is commutative up to a canonical homotopy
$$
\begin{tikzcd}
    \mathcal{S}(\categ C)
    \ar[r,"i_{\mathcal{S}, \mathcal{E}_h^{(2)}\mathcal{S}}"'] \ar[rr, bend left, "i_{\mathcal{S}, \Psh}"]
    & \mathcal{E}_h^{(2)}\mathcal{S}(\categ C)
    \ar[r, "i_{\mathcal{S}, \mathcal{E}_h^{(2)}\mathcal{S}}"']
    &\Psh(\categ C)
\end{tikzcd}
$$
for any lex category $\categ C$.
\end{lemma}

\begin{proof}
This amounts to prove that they are Morita simpler. Actually, $\Psh(\categ C)$ is an infinitary pretopos for any lex category $\categ C$ as a consequence of the two following facts
\begin{enumerate}
    \item the category of sets is an infinitary pretopos
    \item colimits and finite limits in $\Psh(\categ C)$ are computed stalkwise; in other words for any object $X \in \categ C$, the functor $F \in  \Psh(\categ C) \mapsto F(X)$ preserves colimits and finite limits.
\end{enumerate}

The fact that the diagram commutes follows from the fact that all of its functors preserve coproducts.
\end{proof}

Let $\categ C$ be a lex category.

\begin{lemma}\label{lemmafullyfaithcoprod}
The canonical functor 
$$
i_{\mathcal{S}, \Psh}(\categ C):
\mathcal{S}(\categ C)
\to \Psh(\categ C)
$$
is fully faithful.
\end{lemma}

\begin{proof}
Let us denote $i$ this functor. It preserves coproducts as well as the functors
\begin{align*}
    \hom_{\mathcal{S}(\categ C)}(X, -) : \mathcal{S}(\categ C) \to \Set;
    \\
    \hom_{\Psh(\categ C)}(X', -) : \Psh(\categ C) \to \Set ;
\end{align*}
for any $X,X' \in \categ C$. Moreover, the two functors
\begin{align*}
    i_{\mathcal{S}}(\categ C): \categ C \to \mathcal{E}_h^{(2)}\mathcal{S}(\categ C) ;
    \\
    i_{\Psh}(\categ C) : \categ C \to \Psh(\categ C);
\end{align*}
are fully faithful.

Let us consider two objects $X = \coprod_i X_i$ and $Y = \coprod_j Y_j$ in $\mathcal{S}(\categ C)$ where the objects $X_i$ and $Y_j$ belong to $\categ C$. The following diagram is commutative
$$
\begin{tikzcd}
    && \hom_{\mathcal{S}(\categ C)} (X,Y)
  \ar[dd]
  \\
  \prod_i \coprod_i \hom_{\categ C}(X_i,Y_j)
  \ar[rru, "\simeq"] \ar[rrd, "\simeq"']
  \\
  && \hom_{\Psh(\categ C)}(i(X),i(Y)).
\end{tikzcd}
$$
The facts gathered just above tell us that the maps from the left to the right are isomorphic. Subsequently, the vertical map is also isomorphic.
\end{proof}

\begin{lemma}\label{lemmacoprodeq}
Let $X$ be an element of $\mathcal{S}(\categ C)$. Then the lex functor
$$
\hom_{\Psh(\categ C)}(X, -) : \Psh(\categ C) \to \Set
$$
preserves coequalisers of equivalence 2-groupoids.
\end{lemma}

\begin{proof}
One can check that for any coproduct $X = \coprod_i X_i$ of elements of $\categ C$ and any equivalence 2-groupoid $A$ in $\Psh(\categ C)$, the morphism
$$
\hom_{\Psh(\categ C)}(X, A_0)/\hom_{\Psh(\categ C)}(X, A_1) \to \hom_{\Psh(\categ C)}(X, A_0/A_1)
$$
rewrites as the map
$$
(\prod_i A_0(X_i))/ (\prod_i A_1(X_i)) \to \prod_i A_0(X_i)/A_1(X_i)
$$
which is an isomorphism since the quotient functor
$$
m_{\mathcal{E}_h^{(2)}}(\Set) : \mathcal{E}_h^{(2)}(\Set) \to \Set
$$
is an equivalence and so preserves products.
\end{proof}

\begin{lemma}\label{lemmafullyfaith}
The canonical functor 
$$
i_{\mathcal{E}_h^{(2)}\mathcal{S}, \Psh}(\categ C)
:\mathcal{E}_h^{(2)}\mathcal{S}(\categ C)
\to \Psh(\categ C)
$$
is fully faithful.
\end{lemma}

\begin{proof}
This may be proven using the same arguments as those used in Lemma \ref{lemmafullyfaithcoprod}.
Indeed, let us denote $i$ this lex functor.
It preserves coequalisers of equivalence 2-groupoids as well as the lex functors
\begin{align*}
    \hom_{\mathcal{E}_h^{(2)}\mathcal{S}(\categ C)}(X, -) : \mathcal{E}_h^{(2)}\mathcal{S}(\categ C) \to \Set;
    \\
    \hom_{\Psh(\categ C)}(X', -) : \Psh(\categ C) \to \Set ;
\end{align*}
for any $X,X' \in \mathcal{S}(\categ C)$ (Corollary \ref{corcompactobject} and Lemma \ref{lemmacoprodeq}). Moreover, the two functors
$$
\Psh(\categ C) \leftarrow \mathcal S(\categ C) \to \mathcal{E}_h^{(2)}\mathcal{S}(\categ C) .
$$
are fully faithful.

Let us consider two objects $X,Y$ in $\mathcal{E}_h^{(2)}\mathcal{S}(\categ C)$. The following diagram is commutative.
$$
\begin{tikzcd}
    && \hom_{\mathcal{E}_h^{(2)}\mathcal{S}(\categ C)} (X,Y)
  \ar[dd]
  \\
  \varprojlim_{i \in \Delta_{\leq 1}^\op} \varinjlim_{j \in \Delta_{\leq 1}^\op} \hom_{\mathcal S(\categ C)}(X_i,Y_j)
  \ar[rru, "\simeq"] \ar[rrd, "\simeq"']
  \\
  && \hom_{\Psh(\categ C)}(i(X),i(Y))
\end{tikzcd}
$$
The facts gathered just above tell us that the maps from the left to the right are isomorphic. Subsequently, the vertical map is also isomorphic.
\end{proof}

\begin{definition}
An object $Z \in \Psh(\categ C)$ is simple if for any $X,Y \in \categ C$ and any cospan diagram $X \to Z \leftarrow Y$, then the pullback $X \times_Z Y$ belongs to $\mathcal{S}(\categ C)\subseteq \Psh(\categ D)$.
\end{definition}

\begin{remark}
One can show that such a pullback either belongs to $\categ D \subseteq \Psh(\categ D)$ or is the initial object.
\end{remark}

\begin{lemma}
A subobject of a simple object is simple.
\end{lemma}

\begin{proof}
Let us consider a subobject $S \hookrightarrow Z$ of a simple object $Z$. Then, for any cospan diagram $x \to S \leftarrow y$ ($x,y \in \categ D$), its pullback is
$$
x \times_S y = x \times_Z y \in \mathcal{S}(\categ C).
$$
\end{proof}

\begin{proposition}
An object $Z \in \Psh(\categ C)$ is simple if and only if it belongs to the essential image of the canonical fully faithful embedding
$$
\mathcal{E}_h\mathcal{S}(\categ C) \to
\mathcal{E}_h^{(2)}\mathcal{S}(\categ C) \to \Psh(\categ C).
$$
\end{proposition}

\begin{proof}
On the one hand, let $A$ be an element of $\mathcal{E}_h\mathcal{S}(\categ C)$ and let us consider a cospan diagram $X \to A \leftarrow Y$ in $\Psh(\categ C)$, where $X,Y$ actually belong to $\categ C$. Both maps $X \to A$ and $Y \to A$ lift to $A_0$. Since $\Psh(\categ C)$ is a pretopos, we have
$$
X \times_{A} Y = X \times_{A_0} A_1 \times_{A_0} Y \in \mathcal{S}(\categ D).
$$
Hence, the image of $A$ in $\Psh(\categ C)$ is simple.

On the other hand, let $Z$ be a simple object. One can find an object $Y$ in $\mathcal S (\categ C)$ together with an effective epimorphism $Y \to Z$. Then the equivalence relation $R_Y = Y \times_Z Y$ belongs to $\mathcal S (\categ C)$. Moreover, $Z = Y / R_Y$. It is thus the image of the equivalence groupoid $(Y, R_Y) \in \mathcal{E}_h\mathcal{S}(\categ C)$.
\end{proof}

\begin{lemma}\label{lemmaessentiallysurjective}
The embedding
$$
i: \mathcal{E}_h^{(2)}\mathcal{S}(\categ C) \to \Psh(\categ C)
$$
is essentially surjective.
\end{lemma}

\begin{proof}
Let $Z$ be an object of $\Psh(\categ C)$. Since $Z$ is a colimit of representables, one can find a coproduct of representables $Y \in \mathcal{S}(\categ C)$ and an effective epimorphism $Y \to Z$. The equivalence relation $R_Y = Y \times_Z Y$ is simple as a subobject of $Y \times Y \in \mathcal S (\categ D)$. Let $A_0, A_1$ be antecedents in $\mathcal{E}_h\mathcal{S}(\categ C) \subset
\mathcal{E}_h^{(2)}\mathcal{S}(\categ C)$ of the simple objects $(Y, R_Y)$. Let us denote $A$ the quotient in $\mathcal{E}_h^{(2)}\mathcal{S}(\categ C)$ of $A_0$ by its equivalence relation $A_1$. Since the embedding $i$ preserves coequalisers of equivalence groupoids, then $Z \simeq i(A)$.
\end{proof}

\appendix


\section{Some results about pretopos}

In this subsection, we recall some results about pretopos. In particular, the fact that they are cocomplete and that colimits are universal. The reader can refer to \cite{johnstone} for more details.

\subsection{Effective epimorphism}

Let $\categ C$ be a category that admits finite limits.

\begin{proposition}\label{propeffectiveepi}
Let us consider a morphism $p: X \to Y$ in $\categ C$. The following assertions are equivalent.
\begin{enumerate}
    \item there exists a pair of arrows $f,g: U \to X$ so that the map $p: X \to Y$ is the coequaliser of $f,g$;
    \item the following diagram is colimiting
    $$
    X \times_Y X \rightrightarrows X \xrightarrow{p} Y .
    $$
\end{enumerate}
\end{proposition}

\begin{proof}
It is clear that (2) implies (1). Let us suppose (1) and let us denote $R = X \times_Y X$ and $(\pi_1, \pi_2)$ the two maps from $R$ to $X$. Then (2) follows from the fact that, for a map $p: X \to Z$, the following assertions are equivalent:
\begin{enumerate}
\item $p\circ \pi_1 = p\circ \pi_2$
    \item $p\circ f = p\circ g$ 
    \item $p$ factorises as
    $$
    X \xrightarrow{f} Y \xrightarrow{p'} Z .
    $$
\end{enumerate}
\end{proof}

\begin{definition}
A morphism $p$ that satisfies the above conditions of Proposition \ref{propeffectiveepi} is called an effective epimorphism.
\end{definition}

\begin{remark}
Let $F : \categ D \to \categ E$ be a functor from a small category to a category $\categ E$ that admits limits of diagrams indexed by the cocone category $\categ D^{\triangleright}$ and colimits indexed by the cone category $\categ D^{\triangleleft}$.
Such colimits and limits induce an adjunction
$$
\begin{tikzcd}
    \categ C_{/F}
    \ar[rr, shift left, "\varinjlim_{\categ D^{\triangleleft}}"]
    &&
    \categ C_{F/}
    \ar[ll, shift left, "\varprojlim_{\categ D^{\triangleright}}"]
\end{tikzcd}
$$
so that any set $\hom_{\categ C_{/F}}(X, \lim Y)$
either is empty or has a unique element. Thus the adjunction is idempotent.
\end{remark}

\begin{definition}
We say that an equivalence groupoid $X$ of $\categ C$ is effective if the diagram $X_1 \rightrightarrows X_0$, has a coequaliser $X_0/X_1$ and the map
$$
X_1 \to X_0 \times_{X_0/R_X} X_0
$$
is an isomorphism. Equivalently, the limit functor
\begin{align*}
    \text{Effective epimorphisms from }X
    & \to 
    \text{Equivalence relations above }X 
    \\
    Q &\mapsto X \times_Q X
\end{align*}
is an equivalence of categories.
\end{definition}

\begin{lemma}\label{lemmaeffective}
An equivalence groupoid $Z$ of $\categ C$ is effective if and only if it has a coequaliser and, for any cospan diagram $X \to Z \leftarrow Y$ in $\mathcal{E}_h(\categ C)$, the map
$$
X \times^h_Z Y \to X \times^h_{(Z_1/Z_0)} Y
$$
is an isomorphism.
\end{lemma}

\begin{proof}
Let us suppose that $Z$ has a coequaliser $Q$ and that it satisfies the property described above with respect to pullbacks. Then
$$
Z_1 = Z_0 \times_Z^h Z_0 = Z_0 \times_Q Z_0
$$
Hence $Z$ is effective. The converse assertions follows from a straightforward checking.
\end{proof}

\subsection{Structure of a regular category}

\begin{proposition}\label{propmonoeffepi}
Let $\categ C$ be a category that admits finite limits and that is $\mathcal{E}_h$-cocomplete. Then, a morphism of $\categ C$ is an isomorphism if and only if it is an effective epimorphism and a monomorphism.
\end{proposition}

\begin{proof}
Let $f: X \to Y$ be a morphism that is both an effective epimorphism and a monomorphism. Let $R = X \times_Y X$. The morphism $f$ factorises as
$$
X \xrightarrow{f_1} X/R \xrightarrow{f_2} Y .
$$
Since $f$ is a monomorphism, the map $X \to R$ is an isomorphism. Subsequently $f_1$ is invertible. Since $f$ is an effective epimorphism related to the effective equivalence relation $R$, then $f_2$ is an isomorphism.
Finally, $f = f_2 \circ f_1$ is an isomorphism.
\end{proof}

\begin{proposition}
Let $\categ C$ be a category that admits finite limits and that is $\mathcal{E}_h$-cocomplete. Then, any morphism of $\categ C$ may be decomposed, in a natural and essentially unique way, into an effective epimorphism followed by a monomorphism.
\end{proposition}

\begin{proof}
Let us consider a morphism $f: X \to Y$ and let us denote $R = X \times_Y X$. This is an equivalence relation on $X$. The map $f$ decomposes as
$$
X \to X/R \to Y
.$$
The first map is an effective epimorphism. Let us show that the second map is a monomorphism. Let us consider two morphisms $g_1, g_2 : A \to X/R$ whose compositions with $X/R \to Y$ are equal. Let $B$ be the following  pullback
$$
\begin{tikzcd}
    B
    \ar[r] \ar[d]
    & X \times X
    \ar[d]
    \\
    A
    \ar[r]
    &X/R \times X/R .
\end{tikzcd}
$$
Since the two maps $B \rightrightarrows X \to Y$ are equal, then the map $B \to X \times X$ factorises through $R$, and so the two maps from $B$ to $X/R$ are equal.

Besides, the map $X \times X \to X/R \times X/R$ is an effective epimorphism since $X/R \times X/R$ is the quotient of $X\times X$ by the equivalence relation $R \times R$ by $\mathcal{E}_h$-cocompleteness of $\categ C$. Since effective epimorphisms are stable through pullbacks, then the map $B \to A$ is also an effective epimorphism. 

Combining these two results, then $g_1 = g_2$. Finally, the map $X/R \to Y$ is a monomorphism. 

Now, let us consider another factorisation
$$
X \xrightarrow{e} U \xrightarrow{m} Y
$$
of $f$ by an effective epimorphism $e$ followed by a monomorphism $m$. Since $m$ is a monomorphism, then the two maps $R \rightrightarrows X \to U$ are equal. So $e$ factorises in a unique way through the quotient $X/R$. Let us denote $g$ the resulting morphism from $X/R$ to $U$. This is an effective epimorphism as $U$ is the coequaliser of the maps
$$
X \times_U X \rightrightarrows X \to X/R .
$$
This is also a monomorphism since the composite map $X/R \to U \to Y$ is a monomorphism. So $g$ is an isomorphism by Proposition \ref{propmonoeffepi}.
\end{proof}

\subsection{Cocompleteness}

Let $\categ C$ be a pretopos, that is a category with finite limits that is $\mathcal{S}$-cocomplete and $\mathcal{E}_h$-cocomplete. Let $X \in \categ C$ be an object.

\begin{definition}
Let us consider an element $A\in \categ C_{/X \times X}$. We denote $A^{\op}$ the object of this same category whose underlying element of $\categ C$ is the same as $A$ and whose structural map targeting $X \times X$ is the composition
$$
A \to X \times X \xrightarrow{\tau} X \times X
$$
where the endomorphism of $X \times X$ denoted $\tau$ is the commutator of the cartesian monoidal structure on $\categ C$. This defines an endofunctor $(-)^\op$ of $\categ C_{/X \times X}$.
\end{definition}

\begin{definition}
For any two objects $A, B \in\categ C_{/X \times X}$ we define $A \otimes_X B$ as the element of $\categ C_{/X \times X}$ whose underlying object in $\categ C$ is the pullback $A \times_X  B$:
$$
\begin{tikzcd}
    A \times_X B
    \ar[r] \ar[d]
    & A
    \ar[d, "d_1"]
    \\
    B
    \ar[r, "d_0"']
    & X .
\end{tikzcd}
$$
and whose structural map targeting $X \times X$ is the composition
$$
A \times_X B \to A \times B \xrightarrow{(d_0, d_1)} X \times X.
$$
This defines a monoidal structure on $\categ C_{/X\times X}$ whose unit is $X$ and whose unitors and associator proceed from the cartesian monoidal structure on $\categ C$.
\end{definition}

One can notice that the endofunctor $(-)^\op$ of $\categ C_{/X\times X}$ has the canonical structure of a monoidal functor from $(\categ C_{/X\times X}, \otimes_X)$ to the same category equipped with the opposite monoidal structure $(\categ C_{/X\times X}, \otimes_X^\op)$. In particular, we have a canonical isomorphism
$$
(A \otimes_X B)^\op = B^\op \otimes_X A^\op
$$
for any $A,B \in \categ C_{/X\times X}$.

\begin{definition}
Let $A$ be an element of $\categ C_{/X\times X}$. Then we denote $T_X(A)$ the free $\otimes_X$-monoid obtained from $A$, that is
$$
T_X(A) = \coprod_{n \geq 0} A^{\otimes_X n} .
$$
Moreover, we denote
$$
G_X(A) = T_X(A \sqcup A^\op) .
$$
\end{definition}

We have a canonical natural isomorphism of $\otimes_X$-monoids
$$
G_X(A)^\op \simeq G_X(A)
$$
given by the map
$$
G_X(A)^\op = \left(\coprod_{n \geq 0} (A \sqcup A^\op)^{\otimes_X n}\right)^\op
= \coprod_{n \geq 0} (A^\op \sqcup A)^{\otimes_X^\op n}
= \coprod_{n \geq 0} (A^\op \sqcup A)^{\otimes_X n} ,
$$
for any $A \in \categ C_{/X \times X}$.

\begin{lemma}
For any element $A \in \categ C_{/X \times X}$, the 1-coskeletal simplicial object $R_A = (X, G_X(A))$
in $\categ C$ is an equivalence 2-groupoid. 
\end{lemma}

\begin{proof}
The Kan structure is given for instance by the following canonical morphisms in $\categ C_{X \times X}$
\begin{align*}
    G_X(A) \otimes_X G_X(A) & \xrightarrow{m} G_X(A);
    \\
    G_X(A)^\op \otimes_X G_X(A) & \simeq G_X(A) \otimes_X G_X(A)  \xrightarrow{m} G_X(A);
    \\
    G_X(A) \otimes_X G_X(A)^\op & \simeq G_X(A) \otimes_X G_X(A)  \xrightarrow{m} G_X(A).
\end{align*}
where $m$ is the structural map that makes $G_X(A)$ the free $\otimes_X$-monoid from $A \sqcup A^\op$.
\end{proof}

Since $R_A$ is an equivalence 2-groupoid, it admits a quotient in $\categ C$.

\begin{lemma}\label{lemmacoeq}
For any element $A \in \categ C_{/X \times X}$, the quotient of the equivalence 2-groupoid $R_A = (X, G_X(A))$ is also the coequaliser of the pair of maps $A \rightrightarrows X$.
\end{lemma}

\begin{proof}
This follows from the fact that for any morphism $f: X \to Y$, the two induced maps from $A$ to $Y$ are equal if and only if the two induced maps from $G_X(A)$ to $Y$ are equal.
\end{proof}

\begin{proposition}
The pretopos $\categ C$ is cocomplete.
\end{proposition}

\begin{proof}
It has coequaliser of any pair of maps (Lemma \ref{lemmacoeq}) and coproducts.
\end{proof}

\begin{corollary}
The full subcategory of $\categ{C}_{X \times X}$ spanned by equivalence relations is reflective (that is, the inclusion functor has a left adjoint).
\end{corollary}

\begin{proof}
Such a left adjoint sends an object $A$ to the image into $X \times X$ of $G_X (A)$.
\end{proof}

\begin{proposition}
Colimits in $\categ C$ are universal.
\end{proposition}

\begin{proof}
It suffices to show that coproducts and coequalisers are universal. We already know that coproducts are universal. Let us show that it is also the case for coequalisers. Let us consider a pair of maps $X_1 \rightrightarrows X_0$ and a cospan diagram $X_0 \to Z \leftarrow Y$.
One can check that the canonical map
$$
G_{X_0 \times_Z Y}(X_1 \times_Z Y)
\to G_{X_0}(X_1) \times_Z Y
$$
is an isomorphism.
Moreover, since $\categ C$ is a pretopos, then the map
$$
(X_0 \times_Z Y)/(G_{X_0}(X_1) \times_Z Y) \to(X_0 /G_{X_0}(X_1)) \times_Z Y
$$
is an isomorphism.
Hence, the canonical map
$$
(X_1 \times_Z Y)/ (X_0 \times_Z Y)
\to 
(X_1  / X_0) \times_Z Y
$$
rewrites as the following sequence of isomorphisms
\begin{align*}
    (X_1 \times_Z Y)/ (X_0 \times_Z Y)
    & \simeq (X_0 \times_Z Y)/(G_{X_0 \times_Z Y}(X_1 \times_Z Y))
    \\
    & \simeq (X_0 \times_Z Y)/(G_{X_0}(X_1) \times_Z Y)
    \\
    & \simeq (X_0 /G_{X_0}(X_1)) \times_Z Y
    \\
    & \simeq (X_0 /X_1) \times_Z Y .
\end{align*}
It is thus invertible.
\end{proof}


\bibliographystyle{unsrt}
\bibliography{bib}

\end{document}

%% file: macros.tex
\theoremstyle{plain}
\newtheorem{theorem}{Theorem}
\newtheorem*{theorem*}{Theorem}
\newtheorem{lemma}{Lemma}
\newtheorem{proposition}{Proposition}
\newtheorem{corollary}{Corollary}

\theoremstyle{definition}
\newtheorem{definition}{Definition}

\theoremstyle{remark}
\newtheorem{remark}{\sc Remark}
\newtheorem{example}{\sc Example}

\newcommand{\Fun}[2]{\mathsf{Fun}\left({#1}, {#2}\right)}
\newcommand{\Psh}{\mathsf{Psh}}

\newcommand{\categ}[1]{\mathsf{#1}}

\newcommand{\id}{\mathrm{Id}}

\newcommand{\Set}{\mathsf{Set}}
\newcommand{\sSet}{\mathsf{sSet}}

\newcommand{\op}{\mathrm{op}}

\newcommand{\itemt}{\item[$\triangleright$]}

\newcommand{\poubelle}[1]{}